\documentclass[11pt,reqno]{amsart}
\usepackage[a4paper]{anysize}\marginsize{2.5cm}{2.5cm}{2.5cm}{2.5cm}
\pdfpagewidth=\paperwidth \pdfpageheight=\paperheight
\usepackage[english]{babel}
\usepackage[latin1]{inputenc}
\usepackage[T1]{fontenc} 
\usepackage{graphicx}
\usepackage[all]{xy}
\usepackage{latexsym}
\usepackage{amssymb,amsmath,amstext,amsfonts,amsthm}
\usepackage{epsfig}
\usepackage{graphics}
\usepackage{subfigure}
\usepackage{ifthen}
\usepackage{wrapfig}
\usepackage{multicol}
\usepackage{paralist}
\usepackage{verbatim}
\usepackage{epic}
\usepackage{rotating}
\usepackage{xcolor}
\usepackage{mathrsfs}
\usepackage{euscript}
\usepackage{listings}
\usepackage{mathtools}
\usepackage{pgf,tikz}
\usepackage{relsize}
\usetikzlibrary{arrows}
\usepackage[colorlinks=true, urlcolor=black, linkcolor=blue, citecolor=black]{hyperref}

\numberwithin{equation}{section}
\theoremstyle{plain}
\newtheorem*{theorem*}{Theorem}
\newtheorem{theorem}{Theorem}
\numberwithin{theorem}{section}
\newtheorem{proposition}[theorem]{Proposition}
\newtheorem{lemma}[theorem]{Lemma}
\newtheorem{corollary}[theorem]{Corollary}

\theoremstyle{definition}
\newtheorem{definition}[theorem]{Definition}
\newtheorem{remark}[theorem]{Remark}
\newtheorem{example}[theorem]{Example}

\newcommand{\C}{\mathbb{C}}

\newcommand{\PP}{\mathbb{P}}

\newcommand{\R}{\mathbb{R}}

\newcommand{\arxiv}[1]{\href{http://arxiv.org/abs/#1}{{\tt arXiv:#1}}}

\makeatletter
\newcommand*{\rom}[1]{\expandafter\@slowromancap\romannumeral #1@}
\makeatother

\makeatletter
\@namedef{subjclassname@2020}{%
	\textup{2020} Mathematics Subject Classification}
\makeatother

\date{}
\begin{document}

\author{Giorgio Ottaviani}
\address{Dipartimento di Matematica e Informatica ``Ulisse Dini'',  University of Florence, viale Morgagni 67/A, I-50134, Florence, Italy}
\email{giorgio.ottaviani@unifi.it}

\author{Luca Sodomaco}
\address{Department of Mathematics and Systems Analysis, Aalto University, Espoo, Finland}
\email{luca.sodomaco@aalto.fi}

\subjclass[2020]{14C17,14N07,14P05,15A72,58K05,90C26}

\title{The Distance Function from a Real Algebraic Variety}

\begin{abstract}For any (real) algebraic variety $X$ in a Euclidean space $V$ endowed with a nondegenerate quadratic form $q$, we introduce a polynomial $\mathrm{EDpoly}_{X,u}(t^2)$ which, for any $u\in V$, has among its roots the distance from $u$ to $X$. The degree of $\mathrm{EDpoly}_{X,u}$ is the {\em Euclidean Distance degree} of $X$.
We prove a duality property when $X$ is a projective variety, namely $\mathrm{EDpoly}_{X,u}(t^2)=\mathrm{EDpoly}_{X^\vee,u}(q(u)-t^2)$ where $X^\vee$ is the dual variety of $X$. When $X$ is transversal to the isotropic quadric $Q$, we prove that the ED polynomial of $X$ is monic and the zero locus of its lower term is $X\cup(X^\vee\cap Q)^\vee$. 
\end{abstract}

\maketitle

\section{Introduction}
Let $X_{\R}$ be a real algebraic variety in a Euclidean space $V_{\R}$ endowed with a positive definite quadratic form $q\colon V_{\R}\to\R$. Given a data point $u\in V_{\R}$, it is useful for many applications to look for the closest point $x\in X_{\R}$ from $u$. To achieve this goal, we consider the {\it squared Euclidean distance} from $u$ to $X_\R$
\[
t^2(u)\coloneqq \min_{x\in X_{\R}}q(u-x)\,.
\]
Since $X_\R$ is defined by polynomial equations, it turns out that the distance function $t(u)$ depends algebraically on $u$. 

The starting point of this paper is that $t$ satisfies an equation of the form
\begin{equation}\label{eq:algfunction}
\sum_{i=0}^kp_i(u)t^{2i}=0\,,
\end{equation}
where $p_i(u)$ are polynomials on $V_{\R}$. Indeed, the knowledge of (\ref{eq:algfunction}) furnishes several metric properties of $X_\R$ which are not apparent from its equations. On the one hand, the explicit determination of (\ref{eq:algfunction}) with computer algebra systems requires elimination theory and is expensive even in basic examples. On the other hand, it is possible to describe a priori some properties of (\ref{eq:algfunction}). First, the degree in $t^2$ of (\ref{eq:algfunction}), which was called {\em Euclidean Distance degree} in \cite{DHOST}, measures the complexity of computing the distance from $X_\R$. In this paper, we go further in this investigation and we describe completely the geometry of the extreme coefficients $p_0(u)$ and $p_k(u)$ in (\ref{eq:algfunction}).

For example, let $X_\R\subset M_{m,n,\R}$ be the variety of real $m\times n$ matrices of corank one, with $m\le n$. The matrix space $M_{m,n,\R}$ is endowed with the $L_2$-form $q(U)=\mathrm{tr}(UU^\mathsmaller{T})$ for any $U\in M_{m,n,\R}$. In this case, (\ref{eq:algfunction}) reduces to
\begin{equation}\label{eq: matrix case algebraic rel}
\det\left(t^2 I_m-UU^\mathsmaller{T}\right)=0\,.
\end{equation}
Note that the polynomial in \eqref{eq: matrix case algebraic rel} is monic in $t$ and its lowest coefficient is $p_0(U)=\det(UU^\mathsmaller{T})$, which is a sum of squares by the Cauchy-Binet formula and vanishes (on the real numbers) exactly on $X_\R$. In the square case $m=n$, the coefficient $p_0(U)$ simplifies to $\det(U)^2$. The roots of \eqref{eq: matrix case algebraic rel} are all real and correspond to the squared singular values of $U$. This example shows another basic feature of the problem. More precisely, the distance from $U$ to $X_\R$ is computed by the minimum singular value, but this cannot be considered alone without referring to the totality of singular values. Indeed, in the algebraic closure of the coefficients of $U$, the singular values are permuted by each other when $U$ moves along a loop.

More generally, the roots of (\ref{eq:algfunction}) may be complex. For this reason, in this paper we consider the complex vector space $V\coloneqq V_\R\otimes\C$ and define $X$ to be the complex zero set of the equations of $X_\R$.
 
Our description generalizes the fact that \eqref{eq:algfunction} is monic in $t$ to any variety $X$, with mild transversality assumptions, see Corollary \ref{cor: degree constant term}. For most varieties $X$ an additional factor appears in the lowest term, see Corollary \ref{cor:constterm for cones}.
The polynomial on the left-hand side of (\ref{eq:algfunction}) defines an algebraic extension of the field of rational functions. Conversely, such algebraic extension allows to recover $X$. When (\ref{eq:algfunction}) is not monic in $t$, the algebraic function $t$ is not {\it integral}. This means geometrically that a root goes to infinity. 
An interesting example is when $X$ is a cone with vertex in the origin, corresponding to a projective variety of dimension $m$, like in the matrix example mentioned above. When $X$ is smooth and transversal to the isotropic quadric $Q$, a formula for the {\it degree} $k$ of the polynomial in (\ref{eq:algfunction}) (the ED degree in \cite{DHOST} denoted by $\mathrm{EDdegree}(X)$) was proved in 2000 by Catanese and Trifogli \cite{CT} and it is
\begin{equation}\label{eq:ct}
\sum_{i=0}^{m}(-1)^i(2^{m+1-i}-1)\deg c_i(X)\,.
\end{equation}
In \cite{DHOST} it was proved that $\mathrm{EDdegree}(X)=\mathrm{EDdegree}(X^\vee)$,
where $X^\vee$ is the projective dual of $X$. For every $u\in V$, the (complexified) distance from $u$ to $X$ is one of the roots of the polynomial in (\ref{eq:algfunction}), while the other roots correspond to the values assumed on $X$ by the other critical points of the distance function from $u$. Indeed, $\mathrm{EDdegree}(X)$ counts the number of these critical points. We call the left-hand side of (\ref{eq:algfunction}) the {\it ED polynomial} of $X$ and we denote it by $\mathrm{EDpoly}_{X,u}(t^2)$.

A formidable insight into the problem has been achieved by P. Aluffi in \cite[Prop. 2.9]{Alu}, who proved that the same formula (\ref{eq:ct})
holds allowing $X$ to be singular, when Chern classes $c_i(X)$ are replaced by Chern-Mather classes $c_i^\mathsmaller{M}(X)$, with the assumption on transversality refined by using a Whitney stratification of $X$. The work of Aluffi gives new tools to work with dual varieties and we report
a general consequence in Corollary \ref{cor:hypersurface}. A very explicit description of Chern-Mather classes in the toric case has been obtained by Helmer and Sturmfels in \cite{HS} and by Piene in \cite{Pie16}.

The content of the paper is the following.

In Section 2 we give the main definitions and show the first properties of the ED polynomial.

In Section 3 we show that the ED polynomial of a projective variety $X$ and its dual $X^\vee$ are linked by the following formula which enhances \cite[Theorem 5.2]{DHOST}.
\newtheorem*{thm:EDduality}{Theorem \ref{thm:EDduality}}
\begin{thm:EDduality}
{\it Let $X$ be a projective variety and $X^\vee$ its projective dual. Then for any data point $u\in V$
\[
\mathrm{EDpoly}_{X,u}(t^2)=\mathrm{EDpoly}_{X^\vee,u}(q(u)-t^2)\,.
\]
}
\end{thm:EDduality}
Theorem \ref{thm:EDduality} means that projective duality corresponds to variable reflection for the ED polynomial.

In Section 4, we recall the basics of Whitney stratifications and define transversality. We show that with the transversality condition, the ED polynomial is monic.
\newtheorem*{cor: degree constant term}{Corollary \ref{cor: degree constant term}}
\begin{cor: degree constant term}
{\it Let $X\subset\PP(V)$ be a projective variety. If $X$ is transversal to the isotropic quadric $Q$, then for any data point $u\in V$
\[
\mathrm{EDpoly}_{X,u}(t^2)=\sum_{i=0}^{d}p_i(u)t^{2i}\,,
\]
where $d=\mathrm{EDdegree}(X)$ and $p_i(u)$ is a homogeneous polynomial in the coordinates of $u$ of degree $2d-2i$. In particular the ED polynomial of $X$ is monic, up to a scalar factor.}
\end{cor: degree constant term}

In Section 5 we describe the lowest term of the ED polynomial.

\newtheorem*{cor:constterm for cones}{Corollary \ref{cor:constterm for cones}}
\begin{cor:constterm for cones}
{\it Let $X\subset\PP(V)$ be a projective variety. Assume that $X^\vee\cap Q$ is a reduced variety. Then the locus of zeroes $u\in V$ of $\mathrm{EDpoly}_{X,u}(0)$ is
\[
X\cup(X^\vee\cap Q)^\vee\,.
\]
In particular, at least one between $X$ and $(X^\vee\cap Q)^\vee$ is a hypersurface.}
\end{cor:constterm for cones}

In Sections 6 and 7 we recall Chern-Mather classes and we express the ED degree of a projective variety in terms of dual varieties, as in the following result.

\newtheorem*{thm:2EDdegree}{Theorem \ref{thm:2EDdegree}}
\begin{thm:2EDdegree}
{\it Assume that $X^\vee$ is transversal to $Q$. If $X$ is not a hypersurface, then 
\[
2\mathrm{EDdegree}(X)=2\mathrm{EDdegree}(X^\vee)=\mathrm{deg}((X^\vee\cap Q)^\vee)\,.
\]
Otherwise if $X$ is a hypersurface, then
\[
2\mathrm{EDdegree}(X)=2\mathrm{EDdegree}(X^\vee)=2\deg(X)+\mathrm{deg}((X^\vee\cap Q)^\vee)\,.
\]
In particular, $(X^\vee\cap Q)^\vee$ has always even degree.
}
\end{thm:2EDdegree}

In the projective case, Theorem \ref{thm:2EDdegree} leads us to a more precise description of the lowest term of the ED polynomial, with reasonable transversality assumptions. In particular, the factor corresponding to $(X^\vee\cap Q)^\vee$ is always present. If $X$ is a hypersurface, an additional factor corresponding to $X$ appears, see the following Theorem.

\begin{theorem*}
Let $X\subset\PP(V)$ be an irreducible variety and suppose that $X$ and $X^\vee$ are transversal to $Q$. Let $u\in V$ be a data point.
\begin{enumerate}
\item $\mathrm{[{\bf Theorem}\ \ref{thm:eqdual codim X greater than 1}]}$ If $\mathrm{codim}(X)\ge 2$, then $(X^\vee\cap Q)^\vee$ is a hypersurface and
\[
\mathrm{EDpoly}_{X,u}(0)=g
\]
up to a scalar factor, where $g$ is the equation of $(X^\vee\cap Q)^\vee$. Moreover $X\subset (X^\vee\cap Q)^\vee$. 
\item $\mathrm{[{\bf Theorem}\ \ref{thm:eqdual}]}$ If $X$ is a hypersurface, then
\[
\mathrm{EDpoly}_{X,u}(0)=f^2g
\]
up to a scalar factor, where $f$ is the equation of $X$ and $g$ is the equation of $(X^\vee\cap Q)^\vee$.
\end{enumerate}
\end{theorem*}

The transversality assumption is necessary in Corollary \ref{cor: degree constant term}, as it is shown by Example \ref{ex: ED poly dual Veronese} dealing with the Veronese variety, studied by the second author in \cite{Sod}. In this case the degree of $X^\vee$ is greater than $\mathrm{EDdegree}(X)$, so that a formula like in Theorem \ref{thm:2EDdegree} cannot hold. Section 7 closes with other considerations regarding tensor spaces.

Last Section 8 considers the case of matrices. In Theorem \ref{thm: EDpoly matrices} we interpret the polynomial $\det(UU^\mathsmaller{T}-t^2I_m)$ as ED polynomial of the variety of corank one matrices.
Example \ref{exa:ortoinvariant} regards orthogonally invariant matrix varieties studied in \cite{DLOT}.

\section{Definition and first properties}
In the following, we consider $x_1,\ldots,x_n$ as coordinates in $V$. We may assume that the quadratic form $q$ has the Euclidean expression $q(x)=\sum_{i=1}^nx_i^2$. By abuse of notation, we denote with the same letter the quadratic form $q$ and its associated bilinear form, in other words we have $q(x,x)=q(x)$. We denote by $Q$ the isotropic quadric with equation $q(x)=0$, whose only real point is the origin. The quadric $Q$ has no real points in the projective setting. We fix $I_X\coloneqq\langle f_1,\ldots,f_s\rangle\subset\R[x_1,\ldots, x_n]$ a radical ideal. Let $X_{\R}\subset V_{\R}$ be the real zero locus of $I_X$, a real algebraic variety of codimension $c$, and let $X\subset V$ be the complex zero locus of $I_X$. Often it is convenient to see $I_X$ in the larger ring $\C[x_1,\ldots, x_n]$, as the ideal generated by the same polynomials $f_i$.

We denote by $J(f)$ the $s\times n$ Jacobian matrix, whose entry in row $i$ and column $j$ is the partial derivative $\partial f_i/\partial x_j$. The singular locus $X_{\mathrm{sing}}$ of $X$ is defined by
\[
I_{X_{\mathrm{sing}}}\coloneqq I_X+\langle c\times c \mbox{ minors of } J(f)\rangle\,.
\]
The critical ideal $I_u\subset\C[x_1,\ldots, x_n]$ is defined as (see the formula (2.1) in \cite{DHOST})
\begin{equation}\label{eq: critical ideal}
I_u\coloneqq\left(I_{X}+\left\langle (c+1)\times(c+1)-\mbox{minors of}
\begin{pmatrix}
u-x\\
J(f)
\end{pmatrix}
\right\rangle\right):(I_{X_{\mathrm{sing}}})^{\infty}\,.
\end{equation}

\begin{lemma}{\cite[Lemma 2.1]{DHOST}}\label{lem:critical ideal}
For general $u\in V$, the variety of the critical ideal $I_u$ is finite. It consists precisely of the critical points of the squared distance function $d_u$ on the manifold $ X_{\mathrm{sm}}\coloneqq X\setminus X_{\mathrm{sing}}$.
\end{lemma}

\begin{definition}\label{def:psi}
Assume $X\not\subset Q$, which is satisfied if $X_{\R}$ is not empty. The ideal $I_u+(t^2-q(u-x))$ in the ring $\C[x_1,\ldots, x_n,u_1,\ldots, u_n,t]$ defines a variety of dimension $n$ (see \cite[Theorem 4.1]{DHOST}). Since the polynomial ring is a UFD, its projection (eliminating $x_i$) in $\C[u_1,\ldots, u_n,t]$ is generated by a single polynomial in $t^2$ . We denote this generator (defined up to a scalar factor) by $\mathrm{EDpoly}_{X,u}(t^2)$ and we call it the {\em Euclidean Distance polynomial (ED polynomial)} of $X$ at $u$. When $f_i$ have real (resp. rational) coefficients, the elimination procedure gives $\mathrm{EDpoly}_{X,u}(t^2)$
with real (resp. rational) coefficients.
\end{definition}

For any fixed $t\in \R$, the variety defined in $V_{\R}$ by the vanishing of $\mathrm{EDpoly}_{X,u}(t^2)$ is classically known as the {\it $t$-offset} hypersurface of $X_{\R}$, and has several engineering applications, starting from geometric modeling techniques. Farouki and Neff \cite{FN} studied algebraic properties of the $t$-offset in the setting of plane curves.
Recently Horobe\c{t} and Weinstein \cite{HW} have linked the subject to ED degree and to persistent homology and in \cite[Theorem 2.9]{HW} already showed that, in our notations, the degree in $t^2$ of $\mathrm{EDpoly}_{X,u}(t^2)$ equals $\mathrm{EDdegree}(X)$. In this paper we revive the original approach by Salmon \cite{Sal}, who called the offset curves {\it parallel curves}, considered $t$ as a variable and observed in \cite[\S 372, ex. 3]{Sal} that the parallel curve of a parabola drops degree with respect to a general conic, see Example \ref{ex: ED polynomial affine conic}.

\begin{figure}[ht]
\begin{tikzpicture}[line cap=round,line join=round,>=triangle 45,x=1.0cm,y=1.0cm, scale=0.6]
\clip(-7.5,-3) rectangle (7,7);
\draw [line width=1pt] (-2.,2.)-- (-2.0061010188881174,1.9253915519147102)-- (-2.0126412894781573,1.8349178087523623)-- (-2.018091514969857,1.7291834342132328)-- (-2.0213616502648772,1.6299893302641528)-- (-2.022451695363217,1.5318852714134141)-- (-2.023541740461557,1.4250608517759433)-- (-2.022451695363217,1.3280468380235462)-- (-2.020271605166537,1.212502057599343)-- (-2.018091514969857,1.098047322273481)-- (-2.0104611992814774,0.9857726771443025)-- (-2.0006507933964173,0.86368762613005)-- (-1.9875702522163372,0.7307021241323821)-- (-1.971565998789021,0.6045255139060233)-- (-1.9535937676221757,0.47682808193088827)-- (-1.9318379088412576,0.32737479117480434)-- (-1.9006229810251576,0.1703542451905643)-- (-1.871299867016094,0.029414116566156074);
\draw [line width=1pt] (-2.,2.)-- (-1.9915032644961579,2.084383751349717)-- (-1.9833271569059092,2.1736920034894816)-- (-1.9738931866094687,2.262371324276149)-- (-1.9606856281944516,2.3604846153592707)-- (-1.9462202070732426,2.4579689750892895)-- (-1.9317547859520336,2.5497929526414413)-- (-1.9166604334777284,2.6365854793688177)-- (-1.9009371496503273,2.7246358688023884)-- (-1.8799591206362682,2.8330274741138224)-- (-1.856891311291082,2.9452908129272193)-- (-1.8245963782078216,3.079084107129487)-- (-1.8015285688626357,3.175968906379405)-- (-1.7707714897357207,3.286694391236454)-- (-1.7338629947834232,3.3974198760935033)-- (-1.693878791918434,3.520448192601336)-- (-1.6600460048788277,3.6188708458076015)-- (-1.6246753638828757,3.7126799371448236)-- (-1.5800775991488494,3.8234054220018727)-- (-1.5324041265021315,3.9479715924660574)-- (-1.483192799899068,4.055621369410411)-- (-1.4401328891213871,4.147892606791285)-- (-1.387845854605632,4.247853113953899)-- (-1.3386345280025684,4.341662205291121)-- (-1.2894232013995048,4.437009150584692)-- (-1.2371361668837497,4.524666826096523)-- (-1.1833112784116488,4.615400209521049)-- (-1.131024243895894,4.699982177120184)-- (-1.0618208158603357,4.807631954064537)-- (-1.,4.9)-- (-0.9326410835272935,4.989098720913587)-- (-0.8803540490115384,5.0644535647746345)-- (-0.8126884749323259,5.150573386330117)-- (-0.7496364627221507,5.22900393810386)-- (-0.6327595620398746,5.367410794175172)-- (-0.5758589656550823,5.428924952429088)-- (-0.5158826613575985,5.496590526508396)-- (-0.46051991892915195,5.555028976849616)-- (-0.40054361463166815,5.61808098905988)-- (-0.3359537484651472,5.68420870918284)-- (-0.27751529812400916,5.739571451611365)-- (-0.21600113987017963,5.801085609865281)-- (-0.14064629600923848,5.868751183944589)-- (-0.06007250391701196,5.94316180229321)-- (-0.03801637879127646,5.962953579003855);
\draw [line width=1pt] (-0.03801637879127646,5.962953579003855)-- (-0.016667409910643367,5.983278631602183)-- (0.,6.);
\draw [line width=1pt] (0.,6.)-- (0.007440940058896936,6.006056733359961)-- (0.02369806458035484,6.018249576751073)-- (0.015815822388132825,6.012214735072644)-- (0.030471866464295633,6.024653898532262)-- (0.03823094862226417,6.031550860450466)-- (0.04697531105426047,6.039309942608445)-- (0.10654089183360715,6.092910903332557);
\draw [line width=1pt] (0.10654089183360715,6.092910903332557)-- (0.33600770332278884,6.293800992945117)-- (0.6469746574028202,6.568265175603173)-- (0.8864238567355238,6.781756386512943)-- (1.2016280018748116,7.072591807826762)-- (1.4834575904699323,7.348064338032917)-- (1.7475845185077155,7.617749970021379)-- (1.902982648729627,7.782253459436214)-- (2.042598156350875,7.9352234938736395)-- (2.1967822386804277,8.108225318534998)-- (2.3855666859422033,8.32311179548282);
\draw [line width=1pt] (-1.7805615966263373,-0.31883845947713346)-- (-1.7728575551123,-0.34329779264779836)-- (-1.7637310132304043,-0.37253997786127246)-- (-1.7551632392188286,-0.39824329989604584)-- (-1.7443603937259724,-0.43046558041789945)-- (-1.733930060146663,-0.4611978132855631)-- (-1.7249897742215405,-0.4855973436229204)-- (-1.7147456965990056,-0.5131632252254312)-- (-1.7029510756722306,-0.5450676309058872)-- (-1.6877385828668638,-0.5859893043239564)-- (-1.6760044575901407,-0.6152314895374305)-- (-1.6633544518415526,-0.6462810319357901)-- (-1.647303328447496,-0.687185507681995)-- (-1.6307344268794375,-0.7270544270801947)-- (-1.616236638007386,-0.7638166774343789)-- (-1.5981144019173221,-0.8062744877025915)-- (-1.5774032749572489,-0.8539100797108301)-- (-1.5539103713977513,-0.9079445221829706)-- (-1.53483945596237,-0.9493324663193912)-- (-1.5141454838941901,-0.9927492312468128)-- (-1.496697625091607,-1.0304852979594317)-- (-1.4715402472832315,-1.0828288743672576)-- (-1.4498371756034836,-1.1242580763242478)-- (-1.4241965694704295,-1.1725227466924206)-- (-1.3970476923883723,-1.22267275574685)-- (-1.3698988153063152,-1.2705603583777714)-- (-1.342749938224258,-1.3180708932714413)-- (-1.322809517759893,-1.351123675369965)-- (-1.3069438637870714,-1.3772758522482809)-- (-1.288462992126422,-1.4069149860437056)-- (-1.2698077726199175,-1.435682380609853)-- (-1.2492347268090058,-1.4672393407096875)-- (-1.2235704651450579,-1.5057488372600978)-- (-1.2036101240199142,-1.534682175771816)-- (-1.1823679261252658,-1.5645311262617532)-- (-1.157463280317747,-1.598775014247141)-- (-1.1292624313886463,-1.6366813501454094)-- (-1.107653988702711,-1.6645159542832648)-- (-1.086594913203706,-1.6912518240472576)-- (-1.063704613748266,-1.7198189177676881)-- (-1.0360531320060944,-1.7527809489835697)-- (-1.0062041815162008,-1.787574204155889)-- (-0.9820320252912561,-1.8150425635024572)-- (-0.9532818091752233,-1.8461733707619008)-- (-0.9263628170156257,-1.8749235868779752)-- (-0.8932176634041477,-1.9082518628851444)-- (-0.8664817936401936,-1.9349877326491371)-- (-0.8375484551285174,-1.962456091995705)-- (-0.8018395879780309,-1.9948687560246607)-- (-0.767412577597049,-2.0250839513058856)-- (-0.7403104630418079,-2.048157373157003)-- (-0.712260074903584,-2.070395205919518)-- (-0.6896442378245455,-2.0887168967177527)-- (-0.6690323356765611,-2.1047483761662082)-- (-0.6438400108290245,-2.1233563433831653)-- (-0.62,-2.14)-- (-0.5931690847152294,-2.158568342886023)-- (-0.569746277555396,-2.1742575033846205)-- (-0.5469682226226843,-2.1889005386985274)-- (-0.5247818054804327,-2.202951936221974)-- (-0.49386873092889544,-2.221144798278646)-- (-0.4659138453296584,-2.2371190186210903)-- (-0.44210042426364166,-2.250282959458845)-- (-0.41991400712139004,-2.2616719869252173)-- (-0.3925507593126129,-2.275575475001064)-- (-0.366814515427601,-2.287851959153128)-- (-0.34,-2.3)-- (-0.31741275992418744,-2.309742557400181)-- (-0.2953742522295508,-2.3190608525999403)-- (-0.2768855712776745,-2.326160506085471)-- (-0.25558661082111295,-2.333999706809078)-- (-0.2316252803074812,-2.342578454770761)-- (-0.21,-2.35)-- (-0.1893231782895881,-2.356481942846592)-- (-0.1615162021379661,-2.3649127813607937)-- (-0.13193431261496394,-2.3730478009796316)-- (-0.10338778922526687,-2.3799995450175473)-- (-0.07365799025464971,-2.386655470160233)-- (-0.04940084084578795,-2.391684391379151)-- (-0.029285155970146484,-2.3952342181219164)-- (-0.017181955492201847,-2.3970681960307965)-- (-0.00751229028756571,-2.3986985465594874)-- (0.,-2.4)-- (0.01246138989065485,-2.4016990969121728)-- (0.023423829207000698,-2.4030568302219972)-- (0.03343082656458246,-2.404263704275175)-- (0.04816474729710234,-2.4059231560982943)-- (0.06475926552826808,-2.4076328943402956)-- (0.0783365986264946,-2.4087391955557087)-- (0.09095848976595702,-2.409644351095592)-- (0.10398267225647802,-2.4104489337977104)-- (0.11861602015123444,-2.411152943662064)-- (0.13093619277740307,-2.4116558078508876)-- (0.14320607898468937,-2.412058099201947)-- (0.15633083431297515,-2.4122592448774767)-- (0.172573347611963,-2.412309531296359)-- (0.19052559915295136,-2.4121586720397117)-- (0.20968472474711544,-2.4117060942697703)-- (0.22617867014051654,-2.4110020844054167)-- (0.245,-2.41)-- (0.26383047326146847,-2.408524885856892)-- (0.28197433832250585,-2.407129203929118)-- (0.3023668020449538,-2.404880605267704)-- (0.3250078644288124,-2.402011703527279)-- (0.35960070644029346,-2.396862706207764)-- (0.38908959088648465,-2.391446380493149)-- (0.4251984289838615,-2.384024008217566)-- (0.4619090810495281,-2.3751974033493046)-- (0.5018294076127393,-2.3641641472639776)-- (0.5493727111076189,-2.3489181934006176)-- (0.5977184332268834,-2.3314655883201856)-- (0.6462647600022458,-2.311405122710501)-- (0.6982213659312491,-2.2873325639788793)-- (0.7455640647700321,-2.2630594005911604)-- (0.795314019481956,-2.234974748737602)-- (0.8524863464694681,-2.199267119952363)-- (0.908053836208208,-2.161352839950059)-- (0.9700406749420382,-2.114411350423402)-- (1.0234015134637162,-2.070679535394289)-- (1.0693399797098224,-2.0305586041749195)-- (1.1168832832047029,-1.9850213472409217)-- (1.1670344472288365,-1.935070787872807)-- (1.1955203083945443,-1.9045788801460863)-- (1.219793471782225,-1.8768954376047216)-- (1.2518939974923946,-1.8399614077289006)-- (1.2810538404037113,-1.8048615968170758)-- (1.308053694951227,-1.7719217742690554)-- (1.3369435393170686,-1.734661974993426)-- (1.3596234171369819,-1.704152139354686)-- (1.3836532876842706,-1.6720223124430922)-- (1.40876315241346,-1.6366525029857917)-- (1.4333330200516994,-1.6015526920739669)-- (1.4519629196894852,-1.5742828389809336)-- (1.4714028149636964,-1.5448529975240959)-- (1.490032714601482,-1.5154231560672582)-- (1.5092880260065173,-1.4859001832953334)-- (1.5,-1.5)-- (1.520389210688688,-1.4680023549301733)-- (1.5360214911594996,-1.4419485541454464)-- (1.5625284015230498,-1.3989031441532889)-- (1.5867697639922793,-1.3592560560026175)-- (1.6108298694489156,-1.3198887199809326)-- (1.6427049500444193,-1.263638577753485)-- (1.675517533010379,-1.2087946890817236)-- (1.7036426041240589,-1.1591070634474783)-- (1.7284864169411427,-1.115981954406435)-- (1.7547364833139105,-1.069106835883562)-- (1.7833303056128185,-1.020825463805003)-- (1.813799132652641,-0.9702003358003002)-- (1.833486682432217,-0.9378565040195178)-- (1.8522367298413367,-0.906450174609193)-- (1.8705180260652288,-0.8787938546806978)-- (1.8892680734743486,-0.8511375347522028)-- (1.9070806185130125,-0.8239499660089364)-- (1.9248931635516764,-0.7976998996361275)-- (1.9435964809711592,-0.7707398139943259)-- (1.9633273325837743,-0.7437608944423406)-- (1.9832595194169262,-0.7169833101108924)-- (2.0019836949268517,-0.6940310949696444)-- (2.0186945182314138,-0.6734949024748499)-- (2.036009347197586,-0.6539653860827412)-- (2.0513108239583895,-0.6370532275576163)-- (2.072249686894226,-0.6145036828574496)-- (2.096208578138116,-0.589940785951911)-- (2.1298315599677817,-0.5581298211070334)-- (2.1630518713563682,-0.528936214129139)-- (2.1960708475244184,-0.5017559593566168)-- (2.2228484318558244,-0.48162243730289656)-- (2.2522433740542045,-0.46048223914648984)-- (2.2834503332374223,-0.43954337621062084)-- (2.309019906245607,-0.4234365585676447)-- (2.3353948201359396,-0.40813508180681735)-- (2.3625750749084196,-0.39283360504599)-- (2.391164676224658,-0.3771294578440883)-- (2.418344930997138,-0.36303599240648415)-- (2.450357231062503,-0.3477345156456568)-- (2.484584218553774,-0.3322317036642923)-- (2.524851262661152,-0.3147155394775556)-- (2.56572231243014,-0.2986087218345794)-- (2.6051840156553703,-0.2839112507353637)-- (2.653303133363687,-0.26659642176916437)-- (2.7133010290836794,-0.24686557015651858)-- (2.7765202883322626,-0.22753738898494724)-- (2.826048752584337,-0.21324258832680593)-- (2.866315796691715,-0.202370486417797)-- (2.9087975282249987,-0.1919010549498625)-- (2.977855508869151,-0.17438489076312597)-- (3.025571956136394,-0.16351278885411705)-- (3.0805364713429646,-0.15143267562188495)-- (3.159057207352351,-0.1351245227583716)-- (3.232343227628072,-0.12042705165915586)-- (3.3100586227553186,-0.10613225100101453)-- (3.3797206090610885,-0.09344813210717082)-- (3.4487785897052476,-0.08177068931601311)-- (3.512400519394947,-0.07150259306861582)-- (3.588505232757899,-0.05922114461584651)-- (3.6712540083985687,-0.046537025722002794)-- (3.74152000036595,-0.0360675942540683)-- (3.8115846571127645,-0.02620216844774541)-- (3.88547468304981,-0.015531401759273714)-- (3.956747351119875,-0.005867311173488026)-- (4.,0.)-- (4.056609620506188,0.007622148602504496)-- (4.1365397030593405,0.017487574408827387)-- (4.247676744795713,0.03198371028750592)-- (4.375725945057187,0.04748652226887046)-- (4.51223122458121,0.0633920046913094)-- (4.629408322933691,0.07688146446730193)-- (4.905640245510354,0.1070817475478822)-- (5.273133169494088,0.14559636797820494)-- (5.643348934558216,0.18261794448467142)-- (6.027535105851179,0.22033804130258067)-- (6.497639275451477,0.2657418615463603)-- (6.927229266988155,0.3048589989871551)-- (7.514684848910693,0.3579465426568051)-- (8.098716413217973,0.4088154119643845);
\draw [line width=1pt] (-1.7805615966263373,-0.31883845947713346)-- (-1.787266990469414,-0.297649414932973)-- (-1.792094874036429,-0.28102003820211285)-- (-1.7970568654803059,-0.26519530873242336)-- (-1.8015368323981797,-0.2511411586067161)-- (-1.8060334998111258,-0.23393377682782404)-- (-1.8083471699002662,-0.22511735359381138)-- (-1.811825653588652,-0.2124200627969742)-- (-1.8152451457855412,-0.1988649283971524)-- (-1.8188080757435487,-0.1835652879892094)-- (-1.8221687853342348,-0.16897890946545055)-- (-1.8254785874928212,-0.15449852502160932)-- (-1.8291019844924927,-0.1400193068223624)-- (-1.833100537387191,-0.1233586697610905)-- (-1.837432303023114,-0.10536518173491682)-- (-1.8414308559178123,-0.08803811919119403)-- (-1.8450804697244878,-0.07286501419684081)-- (-1.8487857481977452,-0.05804390030378533)-- (-1.8524910266710024,-0.044281437403090956)-- (-1.8561963051442598,-0.030518974502396573)-- (-1.859901583617517,-0.01622718610552164)-- (-1.8633467225857363,-0.002389405080179365)-- (-1.8666138055378099,0.010025510137722708)-- (-1.869322799897498,0.020402262952604503)-- (-1.871299867016094,0.029414116566156074);
\draw [line width=1.2pt,domain=-2.4:-1.55] plot(\x,{(--75.65635265383675--40.009787024954605*\x)/-4.906902163803071});
\draw [line width=1pt,domain=-1.965:-1.931] plot(\x-0.3,{(--75.65635265383675--40.009787024954605*\x)/-4.906902163803071});
\draw [line width=1.2pt,domain=-4.3:-2] plot(\x,{(-35.92415701027282-4.906902163803071*\x)/-40.009787024954605});
\draw [line width=1pt,domain=-2.21:-1.94] plot(\x,{((-35.92415701027282-4.906902163803071*\x)/-40.009787024954605)-0.3});
\draw [line width=1.2pt,dash pattern=on 2pt off 2pt] (-4.298836516034809,0.37066397775764676) circle (2.3448622476979537cm);
\begin{scriptsize}
\draw [fill=black] (-1.9714125577716364,0.6561051786759521) circle (2.5pt);
\draw[color=black] (1,6) node {{\large $X$}};
\draw[color=black] (-1.6,1) node {{\large $x$}};
\draw [fill=black] (-4.298836516034809,0.37066397775764676) circle (2.5pt);
\draw[color=black] (-4.1,0.8) node {{\large $u$}};
\draw[color=black] (-2.4,-2.4) node {{\large $T_xX$}};
\end{scriptsize}
\end{tikzpicture}
\caption{A critical point $x\in X$ for the distance function $d_u$ on $X_{\mathrm{sm}}$. The squared norm $t^2=q(u-x)$ is a root of $\mathrm{EDpoly}_{X,u}$.}\label{fig: critical points}
\end{figure}
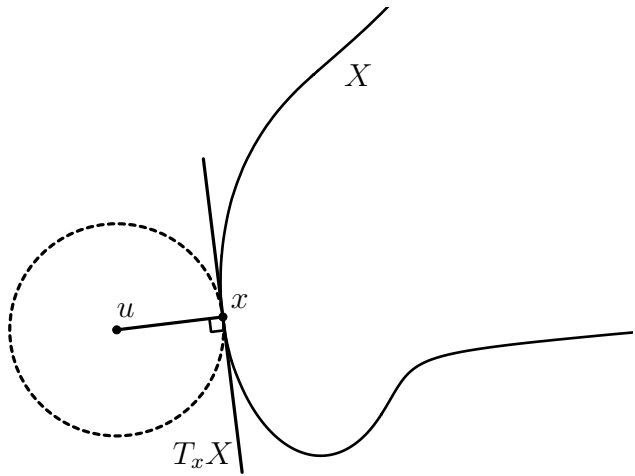

The following Proposition states that the distance function is a root of the ED polynomial and explains the title of this paper.

\begin{proposition}\label{pro:roots charpoly}
For general $u\in V$, the roots of $\mathrm{EDpoly}_{X,u}(t^2)$ are precisely of the form $t^2=q(u-x)$, where $x$ is a critical point of the squared distance function $d_u$ on the manifold $X_{\mathrm{sm}}$. In particular the distance $t$ from $X_{\R}$ to $u\in V_{\R}$ is a root of $\mathrm{EDpoly}_{X,u}(t^2)$.
\end{proposition}

\begin{proof}
It is a consequence of Lemma \ref{lem:critical ideal} and Definition \ref{def:psi}.\qedhere
\end{proof}

The ED polynomial behaves well under union of varieties, as shown by the following.

\begin{proposition}\label{pro:charpoly union}
Assume $X=X_1\cup\cdots\cup X_r$ for some integer $r\geq 0$, where $X_i\subset V$ is a reduced variety for every $i\in\{1,\ldots,r\}$ and $X_i\neq X_j$ for every $i\neq j$. Then
\[
\mathrm{EDpoly}_{X,u}(t^2)=\prod_{i=1}^r\mathrm{EDpoly}_{X_i,u}(t^2)\,.
\]
\end{proposition}
\begin{proof}
For general $u\in V$, the variety of the critical ideal $I_u$ in $V$ with respect to $X$ is the union of the varieties of $I_u$ with respect to its components $X_i$. The conclusion follows by Lemma \ref{pro:roots charpoly}.\qedhere
\end{proof}

Given an affine variety $X\subset V$, as the data point $u$ moves in $V$, the number of complex-valued critical points of $d_u$ remains constant and is equal to $\mathrm{EDdegree}(X)$. On the other hand, the number of real-valued critical points of $d_u$ is constant on the connected components of the complement of a hypersurface $\Sigma_X\subset V$, called the {\em ED discriminant} of $X$. For a precise definition, see \cite[Sections 4, 7]{DHOST}. In particular, if $u$ is close to $\Sigma_X$, then two distinct real (or complex conjugate) roots of $\mathrm{EDpoly}_{X,u}$ tend to coincide. This fact implies the following natural result, which appears essentially in \cite[Proposition 2.13]{HW}, where the discriminant of the ED polynomial was called {\em offset discriminant}. We denote by $\Delta_X(u)\coloneqq\Delta_{t^2}\mathrm{EDpoly}_{X,u}(t^2)$ the discriminant of the ED polynomial of $X$ at $u$, where the operator $\Delta_{t^2}$ computes the discriminant with respect to the variable $t^2$.

\begin{proposition}\label{prop:discriminants}
Given an affine variety $X\subset V$, let $f$ be the equation of the ED discriminant $\Sigma_X$. Then $f(u)$ divides $\Delta_X(u)$. Moreover, if $X$ is symmetric with respect to $s$ affine hyperplanes $L_1,\ldots,L_s$ of equations $l_1,\ldots,l_s$, then the product $l_1(u)\cdots l_s(u)$ divides $\Delta_X(u)$ as well.
\end{proposition}
\begin{proof} By definition of $\Sigma_X$, any point $u\in\Sigma_X$ satisfies $\Delta_X(u)=0$, since two roots in $t^2$ coincide, by Proposition \ref{pro:roots charpoly}.
Let $u\in L_i$ and let $x\in X_{\mathrm{sm}}\setminus L_i$ be a critical point of the distance function $d_u$ (there exists at least one). Denote by $y$ the reflection of $x$ with respect to $L_i$. 
In particular, $y\in X$ as well and $y$ is again a critical point of $d_u$. Since $q(u-x)=q(u-y)$, we have that $u$ is a zero of $\Delta_X(u)$.\qedhere
\end{proof}

\begin{remark}\label{rem:ED discr sphere}
If $X\subset V$ is symmetric to infinite affine hyperplanes of $V$, then there exist $p\in V$ and $r\in\C$ such that the hyperplanes of symmetry of $X$ are exactly the ones containing $p$ and $X$ is the sphere centered in $p$ of radius $r$. In this special case, $\Delta_X(u)$ coincides with the equation of $\Sigma_X$, which in turn is the sphere centered in $p$ of radius zero.
\end{remark}

\begin{remark}\label{rem:discriminants}
Salmon remarks in \cite[\S 372 ex. 3]{Sal} that when $X$ is an ellipse with symmetry axes $L_1$, $L_2$, with the notation of Proposition \ref{prop:discriminants}, $\Delta(u)=f^3l_1^2l_2^2$, of total degree $22$. Here $f$ is a sextic Lam\'e curve, the evolute of the ellipse. It should be interesting to compute in general the exponents occurring in the factorization of $\Delta(u)$. Already for plane cubics, another factor appears, containing points with two critical points at equal distance, not necessarily coming from symmetry.
\end{remark}

A quite important case occurs when $X$ is a cone with the origin as vertex, meaning that if $x\in X$ then $\lambda x\in X$ for all $\lambda\in\C$.
In this case $I_X$ is homogeneous and moreover all coefficients of $\mathrm{EDpoly}_{X,u}(t^2)$ are homogeneous in $u$.
To $X$ is associated a projective variety in $\PP(V)$. Conversely every projective variety in $\PP(V)$ corresponds to a cone in $V$.

\begin{definition}
Let $Z\subset\PP(V)$ be a projective variety. The \emph{ED polynomial} of $Z$ is by definition the ED polynomial of its affine cone in $V$.
\end{definition}

A special case occurs by adding a hyperplane ``at infinity'' at the affine space $V$. We denote by $H_{\infty}$ the projective space $\PP(V)\cong \PP^{n-1}$. In particular, $V\cup H_{\infty}\cong{\PP(\C\oplus V)}=\PP^n$, with coordinates $x_0,x_1,\ldots,x_{n}$, and $H_{\infty}\colon x_{0}=0$. Given $X\subset V$, we indicate by $\overline{X}$ its closure in the Zariski topology of $\PP^n$. Let $X_{\infty}\coloneqq\overline{X}\cap H_{\infty}$ and $Q_{\infty}\coloneqq\overline{Q}\cap H_{\infty}$, the projective varieties consisting of the points at infinity of $X$ and $Q$ respectively.

We know from \cite[Section 6]{DHOST} that in general the ED degree of an affine variety $X\subset V$ is not preserved under the operation of projective closure. Here we stress that the ED degree of $X$ is computed with respect to the quadratic form $q$, whereas the quadratic form $\overline{q}(x)\coloneqq x_0^2+x_1^2+\cdots+x_{n}^2$ is the one used to compute the ED degree of $\overline{X}$. Note that there are infinitely many possible choices of quadratic forms on $\PP^n$ that restrict to $q$ on $V$. This is one of the reasons why $\mathrm{EDdegree}(X)$ and $\mathrm{EDdegree}(\overline{X})$ are not related in general, see Example  \ref{ex: cardioid}.

A positive result is that, under some reasonable transversality assumptions, $\mathrm{EDdegree}(X)$ and $\mathrm{EDdegree}(\overline{X})$ can be related. In particular, in \cite[Lemma 6.7]{DHOST} it is shown that there is a bijection between the critical points of the distance function from the origin $(0,\ldots,0)$ on $X$ and the critical points of the distance function from $u_0\coloneqq(1,0,\ldots,0)$ on $\overline{X}$. This property has a natural interpretation in terms of ED polynomial, as the following result suggests.

\begin{proposition}
Assume that $\mathrm{EDdegree}(X)=\mathrm{EDdegree}(\overline{X})=r$ and that there are $r$ critical points of $d_0$ on X which satisfy $q(x)\neq -1$. Then, up to scalars,
\[
\mathrm{EDpoly}_{X,0}(t^2)=(1+t^2)^r\cdot\mathrm{EDpoly}_{\overline{X},u_0}\left(\frac{t^2}{1+t^2}\right)\,.
\]
\end{proposition}
\begin{proof}
By \cite[Lemma 6.7]{DHOST}, the map
\[
x\mapsto\left(\frac{1}{1+q(x)},\frac{1}{1+q(x)}x\right)
\]
is a bijection between the critical points of $d_0$ on $X$ and the critical points of $d_{u_0}$ on $\overline{X}\setminus X_{\infty}$. Define $r=\mathrm{EDdegree(X)}$ and let $x_1,\ldots,x_r$ be the critical points of $d_0$ on $X$. By Proposition \ref{pro:roots charpoly}, the roots of $\mathrm{EDpoly}_{X,0}(t^2)$ are $t_i^2=q(x_i-0)=q(x_i)$, $i\in\{1,\ldots,r\}$. Let $\tilde{x}_i\in \overline{X}$ be the critical point of $d_{u_0}$ corresponding to $x_i$ for all $i\in\{1,\ldots,r\}$. Hence for every $i$, $\tilde{t}_i^2=\overline{q}(\tilde{x}_i-u_0)$ is a root of $\mathrm{EDpoly}_{\overline{X},u_0}(\tilde{t}^2)$. Then by hypothesis we have
\[
\mathrm{EDpoly}_{X,0}(t^2)=c\cdot(t^2-t_1^2)\cdots(t^2-t_r^2),\quad\mathrm{EDpoly}_{\overline{X},u_0}(\tilde{t}^2)=\tilde{c}\cdot(\tilde{t}^2-\tilde{t}_1^2)\cdots(\tilde{t}^2-\tilde{t}_r^2)
\]
for some scalars $c,\tilde{c}$. Moreover, the following equalities hold true:
\[
\tilde{t}_i^2=\overline{q}(\tilde{x}_i-u_0)=\overline{q}\left(\frac{1}{1+q(x_i)}-1,\frac{1}{1+q(x_i)}x_i\right)=\frac{t_i^2}{1+t_i^2},\quad i\in\{1,\ldots,r\}\,.
\]
From this it follows that
\[
\prod_{i=1}^r(\tilde{t}^2-\tilde{t}_i^2)=\prod_{i=1}^r\left(\frac{t^2}{1+t^2}-\frac{t_i^2}{1+t_i^2}\right)=\prod_{i=1}^r\frac{t^2-t_i^2}{(1+t^2)(1+t_i^2)}=\frac{c'}{(1+t^2)^{r}}\prod_{i=1}^r(t^2-t_i^2)\,,
\]
where $c'=\prod_{i=1}^r(1+t_i^2)$. From the last chain of equalities we obtain the desired identity.\qedhere
\end{proof}

Let $O(V)$ be the orthogonal group of linear transformations of $V$ which preserve $q$. Another reasonable property of the ED polynomial is its behavior under the isometry group $\mathrm{Isom}(V)$, where an isometry is the composition of a translation and an element of $O(V)$.  Fix $g\in\mathrm{Isom}(V)$ and consider the transformed variety $gX\coloneqq\{gx\mid x\in X\}$ of $X$. Then the identity $\mathrm{EDdegree}(gX)=\mathrm{EDdegree}(X)$ holds true. For any data point $u\in V$ and any critical point $x$ for $d_u$ we have that $g\cdot x\in gX$ is a critical point for $d_{g\cdot u}$. Moreover, we have the identity $q(g\cdot u-g\cdot x)=q(u-x)$. The immediate consequence in terms of ED polynomial is that
\[
\mathrm{EDpoly}_{gX,u}(t^2)=\mathrm{EDpoly}_{X,g^{-1}u}(t^2)\,.
\]
In the projective setting we can reduce to the subgroup of isometries which fix the origin, which is precisely $O(V)$.
\begin{proposition}
\noindent
\begin{enumerate}
\item Let $X\subset V$ be an affine variety. Let $G\subset\mathrm{Isom}(V)$ be a group that leaves $X$ invariant. Then the coefficients of $\mathrm{EDpoly}_{X,u}$ are $G$-invariant.
\item Let $X\subset\PP(V)$ be a projective variety. Let $G\subset O(V)$ be a group that leaves $X$ invariant. Then the coefficients of $\mathrm{EDpoly}_{X,u}$ are $G$-invariant.
\end{enumerate}
\end{proposition}
\begin{proof}
The proof is the same in both cases. Let $g\in G$. Since $q(u-x)=q\left(g\cdot u-g\cdot x\right)$, the critical values of the distance function from $u$ coincide with the critical values of the distance function from $g\cdot u$.\qedhere
\end{proof}
Now consider the uniform scaling in $V$ with scale factor $c\in C$. Calling $cX$ the scaling of $X$, for any data point $u\in V$ and any critical point $x$ for $d_u$ we have that $cx\in cX$ is a critical point for $d_{cu}$. Moreover, $q(cu-cx)=c^2q(u-x)$ and this implies that
\[
\mathrm{EDpoly}_{cX,u}(c^2t^2)=\mathrm{EDpoly}_{X,c^{-1}u}(t^2)\,.
\] 
\begin{remark}
There are many meaningful examples with such a $G$-action. If $X$ is the variety of $m\times n$ matrices of bounded rank, the group $G=O(m)\times O(n)$ works. If $X$ is the variety of tensors of format $n_1\times\ldots\times n_k$ of bounded rank, the group $G=O(n_1)\times\ldots\times O(n_k)$ works. In these examples, $X$ and $Q$ are not transversal.
It should be interesting to study the intersection between $X$ and $Q$ when a positive dimensional group $G\subset O(V)$ acts on $X$.
\end{remark}

The simplest varieties to consider are affine subspaces.
\begin{proposition}[The ED polynomial of an affine subspace]\label{prop:affine} Let  $L\subset V_{\R}$ be an affine subspace and let $\pi_{L^\perp}$ be the orthogonal projection on $L^\perp$.
Then $\mathrm{EDdegree}(L)=1$ and for any data point $u\in V$ the ED polynomial of $L$ is
\[
\mathrm{EDpoly}_{L,u}(t^2)=t^2-q\left(\pi_{L^\perp}(u)\right)\,.
\]
\end{proposition}
\begin{proof} The only critical point of $d_u$ from $L$ is $\pi_L(u)$. The equation follows from the identity $\pi_{L^\perp}(u)=u-\pi_{L}(u)$, see also \cite[Example 2.2]{DHOST}.\qedhere
\end{proof}

\begin{remark} Proposition \ref{prop:affine} extends to complex subspaces $L$ such that $L_{\infty}$ is transversal to $Q_{\infty}$.
\end{remark}

The case of linear subspaces (that is, affine subspaces containing the origin) is simpler and it is contained in next Corollary.
This will be generalized to any variety in Theorems \ref{thm:eqdual codim X greater than 1} and \ref{thm:eqdual}.
\begin{corollary}[The ED polynomial of a linear subspace]\label{cor:linear subspace}
Let $L\subset V$ be a linear subspace transversal to $Q$ (this is always the case if $L$ is the complexification of a real subspace).
\begin{enumerate}
\item If $\mathrm{codim}(L)\ge 2$, then the dual variety of $L^\perp\cap Q$ is the quadric hypersurface cut out by a polynomial $g$ and
\[
\mathrm{EDpoly}_{L,u}(t^2)=t^2-g(u)\,.
\]
\item If $L$ is the hyperplane cut out by a polynomial $f$, then 
\[
\mathrm{EDpoly}_{L,u}(t^2)=t^2-f^2(u)\,.
\]
\end{enumerate}
\end{corollary}
\begin{proof} In the following, we interpret $L$, $L^\perp$ and $Q$ as projective varieties of $\PP(V)$. For the notion of dual of a projective variety see Definition \ref{def: dual variety}. It is straightforward to check that the dual variety of $L^\perp\cap Q$ is the join between $L$ and $L^\perp\cap Q$ (see Definition \ref{def: join}), which is a quadric cone with vertex $L$, having equation $q\left(\pi_{L^\perp}(u)\right)$. If $L$ is a hyperplane then $L^\perp\cap Q=\emptyset$ and the quadric cone has rank one.
\end{proof}

\begin{example}
For example, if $L\subset\C^2$ is a point, then the quadric $Z_L$ cut out by $q\left(\pi_{L^\perp}(u)\right)$ is the circumference centered in $L$ of radius zero, namely the union of the lines joining $L$ and the points of $Q_{\infty}$ (the dashed lines in Figure \ref{fig: Z_L, L point}), where in this case $Q_{\infty}=I\cup J$, $I=[1,\sqrt{-1},0], J=[1,-\sqrt{-1},0]$.

More in general, the quadric $Z_L$ has a non-trivial description. We refer to Section \ref{sec: study of the constant term} for a complete study of the lowest term of the ED polynomial of an affine variety. Anyway, when restricting to the real points of $V$, the quadric $Z_L$ restricts to $L$.

\begin{figure}[ht]
\begin{tikzpicture}[line cap=round,line join=round,>=triangle 45,x=1.0cm,y=1.0cm, scale=0.35]
\clip(-9,-5) rectangle (8,8);
\draw [rotate around={27.6251450805318:(-0.21023924747911563,1.3512501369434549)},line width=0.8pt] (-0.21023924747911563,1.3512501369434549) ellipse (8.290629030682343cm and 4.296031611518877cm);
\draw [line width=0.8pt,domain=-1.:7.4] plot(\x,{(--53.53327647191173-7.626882326986989*\x)/6.453515815142847});
\draw [line width=0.8pt, dash pattern=on 2pt off 2pt, domain=-9.:8.] plot(\x,{(--2.075188399764412--1.037594199882206*\x)/8.141060512700282});
\draw [line width=0.8pt, dash pattern=on 2pt off 2pt, domain=-4.85:8.] plot(\x,{(--13.0427926595135--6.52139632975675*\x)/3.5009202489602753});
\draw [fill=black] (1.5009202489602753,6.52139632975675) circle (5pt);
\draw[color=black] (1.5,5.5) node {$I$};
\draw[color=black] (-5.2,4.7) node {$Q$};
\draw[color=black] (5.1,4) node {$H_{\infty}$};
\draw [fill=black] (-2.,0.) circle (5pt);
\draw[color=black] (-2.4,0.8) node {$L$};
\draw [fill=black] (6.141060512700281,1.037594199882206) circle (5pt);
\draw[color=black] (6.1,-0.1) node {$J$};
\draw[color=black] (-7,7) node {{$\PP^2$}};
\end{tikzpicture}
\caption{The dashed lines above form the quadric $Z_L$ when $L\subset\C^2$ is a point.}\label{fig: Z_L, L point}
\end{figure}
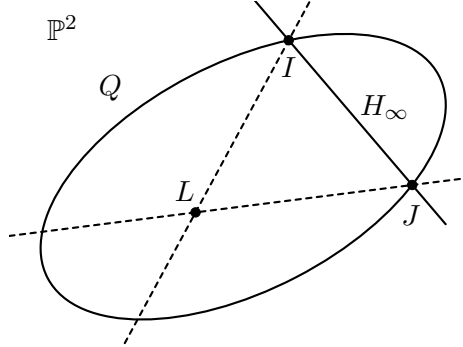
\end{example}

\begin{example}[The ED polynomial of an affine conic]\label{ex: ED polynomial affine conic}
Let $u=(u_1,u_2)\in\C^2$ be a data point and let $C\subset\R^2$ be the real affine conic of equation $ax^2+bxy+cy^2+dx+ey+f=0$. Denote by $A$ the $3\times 3$ symmetric matrix of $C$ and by $B$ the matrix of the circumference $(x-u_1)^2+(y-u_2)^2-t^2=0$. Salmon computes in elegant way $\mathrm{EDpoly}_{C,u}(t^2)=\Delta_{\lambda}\left(\det(A+\lambda B)\right)$ in \cite[\S 373, ex. 3]{Sal}, by invariant theory of pencils of conics.

According to Definition \ref{def:psi}, the following M2 code computes the ED polynomial of $C$ with respect to $u$:
\begin{verbatim}
R = QQ[x,y,a,b,c,d,e,f,u_1,u_2,t];
p = a*x^2+b*x*y+c*y^2+d*x+e*y+f; Ip = ideal p;
ISingp = Ip + ideal(diff(matrix{{x,y}},p));
J = minors(2,diff(matrix{{x,y}},p)||matrix{{x-u_1,y-u_2}});
distance = ideal((x-u_1)^2+(y-u_2)^2-t^2);
sat = saturate(Ip+J+distance,ISingp);
EDpoly = (eliminate({x,y},sat))_0
\end{verbatim}

The output is $\mathrm{EDpoly}_{C,u}(t^2)=c_4t^8+c_3t^6+c_2t^2+c_1t^2+c_0$. Note that this polynomial contains only even powers of $t$ and has degree 4 in $t^2$, according to the fact that the general plane conic $C$ is such that $\mathrm{EDdegree(C)}=4$. In particular, we display the two extreme coefficients of $\mathrm{EDpoly}_{C,u}$:
\begin{align*}
c_4 & = (b^{2}-4 a c)^{2} [(a-c)^{2}+b^{2}],\\
c_0 & = ({a u_1^{2}+b u_1 u_2+c u_2^{2}+d u_1+e u_2+f})^{2} g(u_1,u_2).
\end{align*}
Note that the first factor of $c_0$ is the square of the equation of $C$. Let us concentrate on the factor $g(u_1,u_2)$. As explained in more generality in Proposition \ref{pro: constterm in general}, the locus of zeros of $g$ is the union of the four lines tangent to $C$ and passing through the points of $Q_\infty=\{[0,1,\sqrt{-1}],[0,1,-\sqrt{-1}]\}$. We denote these lines by $L_1^+,L_2^+,L_1^-,L_2^-$, where
\[
L_1^+\cap L_2^+=[0,1,\sqrt{-1}],\quad L_1^-\cap L_2^-=[0,1,-\sqrt{-1}]\,.
\]
The foci of the conic $C$ are by definition the four pairwise intersections $L_1^+\cap L_1^-$, $L_1^+\cap L_2^-$, $L_2^+\cap L_1^-$ and $L_2^+\cap L_2^-$. If $C$ is either an ellipse or a hyperbola, then $C$ has two real foci. Otherwise if $C$ is either a parabola or a circumference, then $C$ has only one real focus. In the case of a circumference, the only real focus coincides with the center. In particular, there exist real solutions of the equation $c_0=0$ outside the real conic $C_\mathbb{R}$ (see also Remark \ref{rmk: real points outside X}).

Now consider the highest coefficients $c_4$. If both of the two factors in $c_4$ do not vanish, then either $C$ is an ellipse or an hyperbola, respectively when $b^2-4ac$ is positive or negative. Moreover, $b^2-4ac=0$ if and only if $C$ is a parabola, whereas $(a-c)^2+b^2=0$ if and only if (over $\R$) $C$ is a circumference:
\begin{enumerate}
\item If $C$ is a parabola then $\mathrm{EDdegree(C)}=3$. Rewriting the equation of $C$ as $(ax+by)^2+cx+dy+e=0$, we get that $\mathrm{EDpoly}_{C,u}(t^2)$ has degree 3 in $t^2$ and its leading coefficient is $-16(bc-ad)^{2} ({a^{2}+b^{2}})^{3}$. In particular, the condition $bc-ad=0$ forces $C$ to be the union of two distinct parallel lines.
\item If $C$ is a circumference, then $\mathrm{EDdegree(C)}=2$. Rewriting the equation of $C$ as $(x~-~a)^2+(y-b)^2-r^2=0$, we get that $\mathrm{EDpoly}_{C,u}(t^2)$ has degree 2 in $t^2$. In particular, it factors~as
\[
\mathrm{EDpoly}_{C,u}(t^2)=[(x-a)^2+(y-b)^2-(t+r)^2][(x-a)^2+(y-b)^2-(t-r)^2]\,.
\]
\end{enumerate}
\end{example}

\section{Projective duality corresponds to variable reflection in the ED polynomial}
In this section we restrict to the case when $X\subset V$ is an affine cone, hence $X$ can be regarded as a projective variety in $\PP(V)$. In this case there is a notion of duality which we recall in the following definition.

\begin{definition}\label{def: dual variety}
Given the above assumptions, we define the {\em dual variety} of $X$ as
\[
X^\vee\coloneqq\overline{\bigcup_{x\in X_{\mathrm{sm}}}N_xX}\,,
\]
where $N_xX\coloneqq\{y\in V\mid y\perp T_xX\}$ is the {\em normal space} of $X$ at the smooth point $x$.
\end{definition}

We remark that this definition is equivalent to the standard one as in \cite{Tev} identifying $V$ and its dual $V^\vee$ via the quadratic form $q$. The variety $X^\vee$ is an irreducible affine cone as well. Note that if $X$ is a linear subspace, then $X^\vee=X^\perp$, the orthogonal subspace with respect to $q$. A fundamental fact showed in \cite[Theorem 5.2]{DHOST} is that, for any data point $u\in V$, there is a bijection $x\to u-x$ from the critical points of $d_u$ on $X$ to the critical ponts of $d_u$ on $X^\vee$. The important consequence of this fact is that $\mathrm{EDdegree}(X)=\mathrm{EDdegree}(X^\vee)$. Moreover, the map is proximity-reversing, meaning that the closer a real critical point $x$ is to $u$, the further $u-x$ is from $u$.

This geometric property has a natural algebraic counterpart in the study of the ED polynomial of $X$, stated in the following proposition.
\begin{theorem}\label{thm:EDduality}
Let $X\subset V$ be an affine cone and $X^\vee$ its dual in $V$. Then
\[
\mathrm{EDpoly}_{X,u}(t^2)=\mathrm{EDpoly}_{X^\vee,u}(q(u)-t^2)\,.
\]
\end{theorem}
\begin{proof}
Consider a data point $u\in V$ and let $x$ be a critical point of $d_u$ on $X$. Then, by Proposition \ref{pro:roots charpoly}, $t_1^2=q(u-x)$ is a root of $\mathrm{EDpoly}_{X,u}$. Applying \cite[Theorem 5.2]{DHOST}, $u-x$ is a critical point of $d_u$ on $X^\vee$. Hence again by Proposition \ref{pro:roots charpoly}, $t_2^2=q(x)$ is a root of $\mathrm{EDpoly}_{X^\vee,u}$. Moreover, by the Pythagorean Theorem we have the equality $q(u)=t_1^2+t_2^2$, thus giving the desired result.\qedhere
\end{proof}

When $X$ is a hypersurface transversal to $Q$ (to be defined in Definition \ref{def:transversal}), after Theorem \ref{thm:eqdual}, the last result allows to write explicitly the equation of $X^\vee$ from $\mathrm{EDpoly}_{X,u}(t^2)$.

\section{The leading term of the ED polynomial}

In this section we study the leading term of the ED polynomial of an affine (resp. projective) variety $X$. We show that with transversality assumptions the leading term is scalar, in other terms the ED polynomial may be written as a monic polynomial, see Proposition \ref{pro:integrality distance function} (resp. Corollary \ref{cor: degree constant term}). In algebraic terms, this implies that the ED polynomial is an integral algebraic function.

We recall from \cite[\S 100]{VdW} that an algebraic function $f(t,u)=\sum_{k=0}^dt^kp_k(u)$ is called {\it integral} if the leading coefficient $p_d(u)$ is constant. The branches of a integral algebraic function are well defined everywhere. Otherwise, if $p_d(u)=0$, then one branch goes to infinity.

To express transversality, we need to recall the Whitney stratification of an algebraic variety. The following definitions are recalled from \cite[\S 4.2]{N}. In the following, $\overline{Y}$ is the Zariski closure of the quasi-projective variety $Y$, which coincides with the closure with respect to the Euclidean topology.
\begin{definition}
Let $X,Y\subset V$ be two disjoint smooth quasi-projective varieties. We say that the pair $(X,Y)$ satisfies the \emph{Whitney regularity condition (a)} at $x_0\in X\cap\overline{Y}$ if, for any sequence $y_n\in Y$ such that
\begin{enumerate}
\item $y_n\rightarrow x_0$,
\item the sequence of tangent spaces $T_{y_n}Y$ converges to the subspace $T$,
\end{enumerate}
we have $T_{x_0}X\subset T$. The pair $(X,Y)$ is said to satisfy the \emph{Whitney regularity condition (a) along $X$}, if it satisfies this condition at any $x\in X\cap\overline{Y}$.
\end{definition}
\begin{definition}
Suppose $X$ is a subset of $V$. A \emph{stratification} of $X$ is an increasing, finite filtration
\[
F_{-1}=\emptyset\subset F_0\subset F_1\subset\cdots\subset F_m=X
\]
satisfying the following properties:
\begin{enumerate}
\item $F_k$ is closed in $X$ for all $k$.
\item For every $k\in\{1,\ldots,m\}$ the set $X_k=F_k\setminus F_{k-1}$ is a smooth manifold of dimension $k$ with finitely many connected components called the $k$-dimensional \emph{strata} of the stratification.
\item (The frontier condition) For every $k\in\{1,\ldots,m\}$ we have $\overline{X}_k\setminus X_k\subset F_{k-1}$.
\end{enumerate}
The stratification is said to satisfy the Whitney condition $(a)$ if, for every $0\leq j<k\leq m$, the pair $(X_j,X_k)$ satisfies Whitney's regularity condition $(a)$ along $X_j$. Note that if $X$ is smooth then the trivial stratification $\emptyset\subset X$ satisfies the Whitney condition (a).
\end{definition}

We recall (see \cite{PP}) that any affine (or projective) variety admits a Whitney stratification, satisfying condition (a) and a stronger condition (b) that we do not use explicitly in this paper.

\begin{definition}\label{def:transversal} We say that a variety $X$ is {\em transversal} to a smooth variety $Y$ (in the applications we will have $Y=Q$)
when there  exists a Whitney stratification of $X$ such that each strata 
is transversal to $Y$.
\end{definition}

If $X$ is smooth and the schematic intersection $X\cap Y$ is smooth then
$X$ is transversal to $Y$ according to Definition \ref{def:transversal}.

If each strata of a Whitney stratification of $X$ is transversal to a smooth variety $Y$, then by \cite[Lemma 1.2]{PP} this stratification induces a Whitney stratification of $X\cap Y$.

\begin{proposition}\label{pro:integrality distance function}
Let $X\subset V$ be an affine variety. If $X_{\infty}$ is transversal to $ Q_{\infty}$, then the ED polynomial of $X$ is an integral algebraic function.
\end{proposition}
\begin{proof} For any non-zero vector $y\in V$, we indicate by $\langle y\rangle\in\PP(V)$ the line spanned by $y$. If the ED polynomial is not an integral algebraic function, then there is a point $u\in V$ that annihilates the leading coefficient of $\mathrm{EDpoly}_{X,u}$. We get a sequence $\{u_k\}\subset V$ such that $u_k\to u$ and a corresponding sequence $\{x_k\}\subset X$ of critical points for $d_{u_k}$ such that $\mathrm{EDpoly}_{X,u_k}(t_k^2)=0$ when $t_k^2=q(x_k-u_k)$ diverges. In particular we have that $u_k-x_k\in \left[\left(T_{x_k}X\right)_\infty\right]^\vee$ for all $k$,
 where the dual is taken in the projective subspace $H_{\infty}$. Up to subsequences, we may assume that
\begin{equation}\label{eq:proposition integrality 1}
\lim_{k\to\infty}\langle x_k\rangle\eqqcolon\langle x\rangle\in X_{\infty}, \textrm{\ for some\ }x\in V\,.
\end{equation} 

We may assume there are two different strata $X_1$ and $X_2$ of $\overline{X}$ such that $x_k\in X_1$ for all $k$ and $\langle x\rangle\in X_2$. 
In the topology of the compact space $\overline{X}=X\cup X_\infty$ we still have $x_k\to \langle x\rangle\in X_{\infty}$, more precisely in $\PP(\C\oplus V)$
we have $\left[(1,x_k)\right]\to\left[(0, x)\right]$.
We may assume (up to subsequences) that $\{T_{x_k}X_1\}$ has a limit. Hence by Whitney condition $(a)$ we have $T_{\langle x\rangle}X_2\subset\lim_{k\to\infty} T_{x_k}X_1$.
From this and from (\ref{eq:proposition integrality 1}) we have immediately that
\begin{equation}\label{eq:proposition integrality 2}
\langle x\rangle\in T_{\langle x\rangle}(X_2)_{\infty}\subset\lim_{k\to\infty}(T_{x_k}X_1)_{\infty}\,.
\end{equation}
Since $\{x_k\}$ diverges we get
\[
\langle x\rangle= \lim_{k\to\infty}\langle u-x_k\rangle = \lim_{k\to\infty}\langle u_k-x_k\rangle\,.
\]
By construction, we have that $\langle u_k-x_k\rangle\in\left[(T_{x_k}X_1)_{\infty}\right]^\vee$ for all $k$.
This fact and relation (\ref{eq:proposition integrality 1}) imply that $\langle x\rangle=\lim_{k\to\infty}\langle u_k-x_k\rangle\in \lim_{k\to\infty}\left[(T_{x_k}X_1)_{\infty}\right]^\vee\subset
\left[\lim_{k\to\infty}(T_{x_k}X_1)_{\infty}\right]^\vee$.

In particular $\langle x\rangle\in \left[\lim_{k\to\infty}(T_{x_k}X_1)_{\infty}\right]\cap \left[\lim_{k\to\infty}(T_{x_k}X_1)_{\infty}\right]^\vee$, hence $\langle x\rangle\in Q_{\infty}$.

We now show that $T_{\langle x\rangle}\left(X_2\cap H_{\infty}\right)\subset T_{\langle x\rangle}Q_{\infty}$, where $T_{\langle x\rangle}Q_{\infty}=\{\langle y\rangle\in H_{\infty}\mid q(\langle x\rangle,\langle y\rangle)=0\}$.
Pick a non-zero vector $v\in V$ such that $\langle v\rangle\in T_{\langle x\rangle}\left(X_2\cap H_{\infty}\right)$. We show that $q(\langle x\rangle,
\langle v\rangle)=0$, where $q$ is the quadratic form in $H_\infty=\PP(V)$, which is the quadratic form defined on $V$.

Pick a sequence $\langle v_k\rangle\to\langle v\rangle$, where
$\langle v_k\rangle\in (T_{x_k}X_1)_{\infty}$.
Since $q(u_k-x_k,v_k)=0$, at the limit we get $q(\langle x\rangle,
\langle v\rangle)=0$,  this contradicts the transversality between $X_{\infty}$ and $ Q_{\infty}$, as we wanted.\qedhere
\end{proof}

\begin{remark}
The transversality conditions stated in Proposition \ref{pro:integrality distance function} are sufficient for the integrality of the distance function, but not necessary: for example, the parabola studied in Example \ref{ex: ED polynomial affine conic} is in general transversal to $Q_\infty$, but not transversal to $H_\infty$. Nevertheless, its ED polynomial is monic.

On the other hand, the cardioid studied in Example \ref{ex: cardioid} is singular at $Q_\infty$ (and consequently is not transversal to $H_\infty$), and its ED polynomial has a leading coefficient of positive degree. It should be interesting to find general necessary conditions for the integrality of the distance function. 
\end{remark}

\begin{corollary}\label{cor: degree constant term}
Let $X\subset \PP(V)$ be a projective variety. If $X$ is transversal to $Q$, then for any data point $u\in V$
\[
\mathrm{EDpoly}_{X,u}(t^2)=\sum_{i=0}^{d}p_i(u)t^{2i}\,,
\]
where $d=\mathrm{EDdegree}(X)$ and $p_i(u)$ is a homogeneous polynomial in the coordinates of $u$ of degree $2d-2i$. In particular, $p_d(u)=p_d\in\C$, $\mathrm{deg}(p_0)=2d$
and the ED polynomial of $X$ is an integral algebraic function.
\end{corollary}
\begin{proof} For affine cones, the assumption that $X$ is transversal to $Q$ is equivalent to $X_\infty$ transversal to $Q_\infty$.\qedhere
\end{proof}

Another consequence of Proposition \ref{pro:integrality distance function} and \cite[Theorem 2.9]{HW}, valid for projective varieties, is the following.

\begin{corollary}
Let $X\subset \PP(V)$ be a projective variety. If $X$ is transversal to $Q$, then the degree of the $t$-offset of $X$, denoted by $\mathcal{O}_t(X)$, is
\[
\deg(\mathcal{O}_t(X))=2\mathrm{EDdegree}(X)\,.
\]
\end{corollary}

\section{The lowest term of the ED polynomial}\label{sec: study of the constant term}
In the last section we gave some geometric conditions on $X$ that affect the shape of its ED polynomial. In particular, we have obtained some useful information about the degrees of the coefficients of $\mathrm{EDpoly}_{X,u}$ for a general data point $u\in V$. In this section, our aim is to describe completely the locus of points $u\in V$ such that the lowest term of the ED polynomial of $X$ at $u$ vanishes. First we recall a definition.

\begin{definition}\label{def: join}
Let $X, Y\in\PP(V)$ be two disjoint projective varieties. The union
\[
J(X,Y)=\bigcup_{x\in X,\ y\in Y}\langle x,y\rangle\subset\PP(V)
\]
of the lines joining points of $X$ to points of $Y$ is again a projective variety (see \cite[Example 6.17]{H}) and is called the \emph{join of $X$ and $Y$}.
Note that $J(X,\emptyset)=X$.
\end{definition}

\begin{proposition}\label{pro: constterm in general}
Consider the variety $X\subset V$. Given any point $p\in X_{\mathrm{sm}}$, we consider $J_{X,p}=J(p,\left[(T_pX)_{\infty}\right]^\vee\cap Q_{\infty})$, where the dual of $(T_pX)_{\infty}$ is taken in $H_{\infty}$.
Then the zero locus $u\in V$ of $\mathrm{EDpoly}_{X,u}(0)$ is $J_{X}\cap V$, where
\[
J_X=\overline{\bigcup_{p\in X_{\mathrm{sm}}}J_{X,p}}\,.
\]
\end{proposition}

\begin{proof}
The variety $\left[(T_pX)_{\infty}\right]^\vee\cap Q_\infty\subset H_\infty$ parametrizes all the directions in $V$ corresponding to isotropic vectors $w$ such that $w\in N_pX$. In particular, for all $p\in X_\mathrm{sm}$, the variety $J_{X,p}$ is the union of the lines passing from $p$ and generated by isotropic vectors $w$ such that $w\in N_pX$.

\medskip
Now pick a point $u\in J_X\cap V$. Then $u$ is a limit of points $u_k$, such that there exist $p_k\in X_\mathrm{sm}$ with $u_k\in J_{X,p_k}$. Hence $u_k=p_k+w_k$ for some vector $w_k\in V$ whose direction belongs to $\left[(T_{p_k}X)_\infty\right]^\vee\cap Q_\infty$. In particular $q\left(w_k\right)=q(u_k-p_k)=0$ and $w_k\in N_{p_k}X$, that is, $p_k$ is a critical point on $X$ of the squared distance function $d_{u_k}$. In particular, $u_k$ is a zero of $\mathrm{EDpoly}_{X,u}(0)$, hence $u$ is a zero of $\mathrm{EDpoly}_{X,u}(0)$ as we wanted.

\medskip
Conversely, let $u\in V$ be a general zero of $\mathrm{EDpoly}_{X,u}(0)$. Then there exist $p\in X_\mathrm{sm}$ such that $u-p\in N_pX$ and $0=q(u-p)$. Define $w=u-p$. Then the direction corresponding to $w$, when nonzero, is represented by a point of $\left[(T_pX)_{\infty}\right]^\vee\cap Q_{\infty}$, namely $u\in J_{X,p}\subset J_X\cap V$. When $w=0$, we have $u=p\in J_{X,p}$. Conclusion follows by taking closures.\qedhere
\end{proof}

\begin{remark}\label{rmk: real points outside X}
As showed in Example \ref{ex: ED polynomial affine conic}, the real foci of a real conic $C$ are zeros of the lowest coefficient of the ED polynomial of $C$. Therefore, given a real affine variety $X_\mathbb{R}\subset V_\mathbb{R}$, the hypersurface $J_X\cap V$ might contain real points which do not necessarily belong to $X_\mathbb{R}$. The reason is that there could exist real data points $u$ admitting (non real) critical points $x$ for the function $\delta_u$ such that $q(u-x)=0$. This is essentially the price to pay for describing the distance function with algebraic tools and, above all, for allowing non real solutions to the problem. For a concrete example, let $C_\mathbb{R}$ be the real ellipse of equation
\[
\frac{x_1^2}{a^2}+\frac{x_2^2}{b^2}-1=0,\quad a\ge b>0\,.
\]
The two foci of $C_\mathbb{R}$ are $u^\pm=(\pm c,0)$, where $c^2=a^2-b^2$. On the one hand, the only real critical points of $\delta_{u^+}$ and $\delta_{u^-}$ on $C_\mathbb{R}$ are the points $(\pm a,0)$. On the other hand, with straightforward computations one verifies that the non real points $z^\pm=\left(\frac{a^2}{c},\pm\sqrt{-1}\frac{b^2}{c}\right)$ are critical for $\delta_{u^+}$ on $C$, as well as $w^\pm=\left(-\frac{a^2}{c},\pm\sqrt{-1}\frac{b^2}{c}\right)$ are critical for $\delta_{u^-}$ on $C$. Note that, for instance $T_{z^\pm}C\colon x_1\pm\sqrt{-1}x_2-c=0$ and coincides with the normal space $N_{z^\pm}C$, with $u^+-z^\pm\in N_{z^\pm}C$. In addition, we have that $q(u^+-z^\pm)=q(u^--w^\pm)=0$, thus confirming the fact that $u^\pm$ are real points of $J_C\cap V$.
\end{remark}

\begin{proposition}\label{pro: JXV}
Let $X\subset V$ be an irreducible cone. Assume that $X^\vee\cap Q$ is a reduced variety. Then 
\[
J_X\cap V=X\cup(X^\vee\cap Q)^\vee\,.
\] 
\end{proposition}

\begin{proof}
Suppose that there exists a point $p\in X_\mathrm{sm}$ such that $u\in J_{X,p}$. If $u=p$, then in particular $u\in X$. If $u\neq p$, then $u-p$ is a nonzero element of $X^\vee\cap Q$. We show that necessarily $u\in (X^\vee\cap Q)^\vee$, where $X^\vee\cap Q$ is reduced by hypothesis. By definition, we have that $(X^\vee\cap Q)^\vee=\overline{S}$, where
\[
S = \bigcup_{z\in(X^\vee\cap Q)_\mathrm{sm}}N_z(X^\vee\cap Q)\,.
\]
On the one hand, by construction $p$ is critical for $\delta_u$ on $X$. On the other hand, since $u$ is general, $u-p$ is critical for $\delta_u$ on $X^\vee$ by \cite[Theorem 5.2]{DHOST}.

In order to prove that $u\in (X^\vee\cap Q)^\vee$, it remains to show that $u\in N_{u-p}(X^\vee\cap Q)$. Indeed, pick a vector $y\in T_{u-p}(X^\vee)$ such that $\langle y,u-p\rangle=0$. The condition that $u-p$ is critical for $\delta_u$ implies that $p$ is orthogonal to any tangent vector to $X^\vee$ at $u-p$, so we have $\langle y, p\rangle=0$. Then
\[
\langle y,u\rangle=\langle y,u-p\rangle+\langle y,p\rangle=0+0=0\,,
\]
thus proving our claim. By taking closures we get 
\[
J_X\cap V\subset X\cup(X^\vee\cap Q)^\vee\,.
\]
On the other hand, suppose that $u\in X\cup(X^\vee\cap Q)^\vee$. If $u\in X$, then clearly $u\in J_X$. Now assume that $u\in S$. Then there exists a smooth $z\in X^\vee\cap Q$ such that $u\in N_z(X^\vee\cap Q)$. In particular, $z$ is critical for $\delta_u$ on $X^\vee$. By \cite[Theorem 5.2]{DHOST}, $u-z$ is critical for $\delta_u$ on $X$. This means that $z$ is an element of $\left[(T_{u-z}X)_{\infty}\right]^\vee\cap Q_\infty$. In particular, $u\in J_{X,u-z}$. By taking the Zariski closure, we have that $ X\cup(X^\vee\cap Q)^\vee\subset J_X\cap V$.\qedhere
\end{proof}

\begin{corollary}\label{cor:constterm for cones}
Let $X\subset V$ be an affine cone such that $X^\vee\cap Q$ is a reduced variety. Then the locus of zeros $u\in V$ of $\mathrm{EDpoly}_{X,u}(0)$ is
\[
X\cup(X^\vee\cap Q)^\vee\,.
\]
In particular, at least one of $X$ and $(X^\vee\cap Q)^\vee$ is a hypersurface.
\end{corollary}

In the next sections, we will show an improvement of the last corollary applying the theory of Chern-Mather classes, which will need an additional transversality assumption.

\section{The ED degree in terms of dual varieties}\label{sec: EDpoly in terms of dual varieties}

Let $Y\subset\PP(V)$ be an irreducible projective variety. We define the {\em conormal variety} of $Y$ as
\[
\mathcal{N}_Y\coloneqq\overline{\{(y,x)\in V\times V\mid y\in Y_{\mathrm{sm}} \mathrm{\ and\ } x\in N_yY\}}\,.
\]
The variety $\mathcal{N}_Y$ is a cone in $V\times V$ and hence it may be regarded as a subvariety of $\PP(V)\times\PP(V)$. An important property of the conormal variety is the {\em Biduality Theorem} \cite[Chapter 1]{GKZ}, stating that $\mathcal{N}_Y=\mathcal{N}_{Y^\vee}$ and implying that $(Y^\vee)^\vee=Y$.

The {\em polar classes} of $Y$ are by definition the coefficients $\delta_i(Y)$ of the class
\[
[\mathcal{N}_{Y}]=\delta_0(Y)s^{n-1}t+\delta_1(Y)s^{n-2}t^2+\cdots+\delta_{n-2}(Y)st^{n-1}\,.
\]
It is shown in \cite[Theorem 5.4]{DHOST} that, if $\mathcal{N}_Y$ does not intersect the diagonal $\Delta(\PP(V))\subset\PP(V)\times\PP(V)$, then
\begin{equation}\label{eq: EDdegree polar classes}
\mathrm{EDdegree}(Y)=\delta_0(Y)+\cdots+\delta_{n-2}(Y)\,.
\end{equation}

A sufficient condition for $\mathcal{N}_{Y}$ not to intersect $\Delta(\PP(V))$ is furnished in the following result. 
\begin{proposition}\label{pro: Whitney trans. implies not meeting of diagonal}
Assume that $Y^\vee$ (or $Y$) is transversal to $Q$, according to definition \ref{def:transversal}. Then $\mathcal{N}_{Y}$ does not intersect $\Delta(\PP(V))$.
\end{proposition}
\begin{proof} By Biduality Theorem it is enough to prove the result for $Y^\vee$.
Suppose that $(y,y)\in\mathcal{N}_{Y}$ for some $y\in Y$. By definition of $\mathcal{N}_{Y}$ and by hypothesis, there exists a sequence of vectors $(y_i,x_i)$ and pairs $(Y_1,Y_2), (Y'_1,Y'_2)$ satisfying the Whitney regularity condition $(a)$ along $Y$ and $Y^\vee$ respectively such that
\begin{enumerate}
\item $(y_i,x_i)\to(y,y)$,
\item $y_i\in Y_1$, $x_i\in Y'_1\cap(T_{y_i}Y_1)^\vee$ for all $i$, and
\item $y\in Y_2\cap Y'_2$.
\end{enumerate}
In particular, point $(2)$ says that $q(y_i,x_i)=0$ for all $i$, hence taking the limit we find $y\in Q$.

Now take a vector $v\in T_yY'_2$. We show that $v\in T_yQ$, obtaining that $T_yY'_2\subset T_yQ$, and thus contradicting the transversality assumption. By the Whitney condition $(a)$ we have that 
$T_yY'_2\subset\lim_{i\to\infty}T_{x_i}Y'_1$. This means that there exists a sequence $\{v_i\}$ with $v_i\in T_{x_i}Y'_1$ for all $i$ such that $v_i\to v$. From point $(2)$ we have $y_i\in(T_{x_i}Y'_1)^\vee$ 
for all $i$, hence $q(y_i,v_i)=0$ for all $i$. Finally taking the limit we have $q(y,v)=0$, that is $v\in T_yQ$.\qedhere
\end{proof}

In the following, we will assume the hypothesis of Proposition \ref{pro: Whitney trans. implies not meeting of diagonal}. If we additionally assume that $Y$ is smooth, the identity (\ref{eq: EDdegree polar classes}) allows us to express the ED degree of $Y$ in terms of Chern classes. The following formula, due essentially to Catanese and Trifogli \cite{CT}, is shown in \cite[Theorem 5.8]{DHOST}
\begin{equation}\label{eq: EDdegree Chern classes smooth case}
\mathrm{EDdegree}(Y)=\sum_{i=0}^m(-1)^i\left(2^{m+1-i}-1\right)c_i(Y)\cdot h^{m-i}\,,
\end{equation}
where $h$ denotes the hyperplane class and $m=\dim(Y)$.

Now we weaken our assumptions, allowing $Y$ to be singular. We recall that the Nash blow-up $\tilde Y$ is the graph of the Gauss map which sends every smooth point of $Y$ to its tangent space in the Grassmannian. The universal bundle of the Grassmannian restricts to a locally free  sheaf on $\tilde Y$. The push-forward to $Y$ of its Chern classes are called the {\em Chern-Mather classes} of $Y$ and are denoted as $c_i^\mathsmaller{M}(Y)$ (see \cite{Alu}). They agree with Chern classes if $Y$ is smooth.
We use a slightly different convention than in \cite{Alu}, for us $c_i^\mathsmaller{M}(X)$ is the component of dimension $\dim X-i$ (as with standard Chern classes), while in \cite{Alu} is the component of dimension $i$. 

Aluffi proves in \cite[Prop. 2.9]{Alu} the following generalization of (\ref{eq: EDdegree Chern classes smooth case}), with the same hypotheses of \cite[Theorem 5.4]{DHOST}:
\begin{equation}\label{eq:AluED}
\mathrm{EDdegree}(Y)=\sum_{i=0}^m(-1)^i\left(2^{m+1-i}-1\right)c_i^\mathsmaller{M}(Y)\cdot h^{m-i}\,.
\end{equation}
The following formula gives the degree of the polar classes $\delta_i(Y)$. It is classical in the smooth case, and it is due to 
Piene (\cite[Theorem 3]{Pie88} and \cite{Pie78}) in the singular case, see also \cite[Prop. 3.13]{Alu}.
\begin{equation}\label{eq: Pie78}
\delta_i(Y)=\sum_{j=0}^{m-i}(-1)^{j}\binom{m+1-j}{i+1}c_{j}^\mathsmaller{M}(Y)\cdot h^{m-j}\,.
\end{equation}
The right-hand side in (\ref{eq: Pie78}) is always a nonnegative integer. The defect of $Y$, namely the difference $\mathrm{codim}(Y^\vee)-1$, equals the minimum $i$ such that $\delta_i(Y)\neq 0$. For example, when $Y^\vee$  is a hypersurface, then
\begin{equation}\label{eq:degdual}
\mathrm{deg}(Y^\vee)=\delta_0(Y)=\sum_{j=0}^m(-1)^{j}\left(m+1-j\right)c_{j}^\mathsmaller{M}(Y)\cdot h^{m-j}\,.
\end{equation}
The formula (\ref{eq:degdual}) can be applied in the case when $Y=X^\vee\cap Q$, once we know the Chern-Mather classes of $X^\vee\cap Q$, where $X\subset\PP(V)$ is a projective variety. Since $X^\vee\cap Q$ is a divisor in $X^\vee$ with normal bundle ${\mathcal O}(2)$, these can be computed by a result of Pragacz and Parusinski in \cite{PP}. We need the assumption that $X^\vee$ is transversal to $Q$, according to Definition \ref{def:transversal}.
Denote by $c_{X^\vee}^\mathsmaller{M}=\sum c_i^\mathsmaller{M}(X^\vee)$ the Chern-Mather class of $X^\vee$. The equation displayed in three lines just after \cite[Lemma 1.2]{PP} shows that

\begin{equation}\label{eq:PP}
c_{X^\vee\cap Q}^\mathsmaller{M}=\frac{1}{1+2h}\sum_{i\ge 0} 2h\cdot c_i^\mathsmaller{M}(X^\vee)\,,
\end{equation}
hence
\begin{equation}\label{eq:chernxq}
c_j^\mathsmaller{M}(X^\vee\cap Q)\cdot h^{m-1-j}=2\sum_{i=0}^j(-1)^{j-i}2^{j-i}c_i^\mathsmaller{M}(X^\vee)\cdot h^{m-i}\,.
\end{equation}
\begin{lemma}\label{lem:2sum}
The following identity holds true:
\[
2^{m+1-i}-1=(m+1-i)+\sum_{j=i}^m(m-j)2^{j-i}\,.
\]
\end{lemma}
\begin{proof}
Using the identities ($r\neq 1$)
\[
\sum_{k=s}^mr^k = \frac{r^s-r^{m+1}}{1-r}\quad\mbox{and}\quad
\sum_{k=1}^mkr^{k-1} = \frac{1-r^{m+1}}{(1-r)^2}-\frac{(m+1)r^m}{1-r}\,,
\]
we have that
\[
\sum_{j=i}^m(m-j)2^{j-i} = \frac{1}{2^i}\left(m\sum_{j=i}^m2^j-2\sum_{j=i}^mj2^{j-1}\right) = \frac{2^{m+1}-(m+2-i)2^i}{2^i} = 2^{m+1-i}-(m+2-i)\,.\qedhere
\]
\end{proof}

\begin{theorem}\label{thm:2EDdegree}
Assume that $X^\vee$ is transversal to $Q$, according to definition \ref{def:transversal}. If $X$ is not a hypersurface, then 
\begin{equation}\label{eq: 2EDdegree X not hypersurface}
2\mathrm{EDdegree}(X)=2\mathrm{EDdegree}(X^\vee)=\mathrm{deg}((X^\vee\cap Q)^\vee)\,.
\end{equation}
Otherwise if $X$ is a hypersurface, then
\begin{equation}\label{eq: 2EDdegree X hypersurface}
2\mathrm{EDdegree}(X)=2\mathrm{EDdegree}(X^\vee)=2\deg(X)+\mathrm{deg}((X^\vee\cap Q)^\vee)\,.
\end{equation}
In particular, $(X^\vee\cap Q)^\vee$ has always even degree.
\end{theorem}
\begin{proof} If $X$ is not a hypersurface, then $(X^\vee\cap Q)^\vee$ is a hypersurface by Corollary \ref{cor:constterm for cones}. Hence we can apply the identity (\ref{eq:degdual}) and we obtain that (here $m=\dim(X^\vee)$)

\begingroup
\allowdisplaybreaks
\begin{align}\label{eq: computation difference Theorem 2EDdegree}
\begin{split}
\deg((X^\vee\cap Q)^\vee) & =\sum_{j=0}^{m-1}(-1)^j\left(m-j\right)c_j^\mathsmaller{M}(X^\vee\cap Q)h^{m-1-j}\\
 & =\sum_{j=0}^m(-1)^j\left(m-j\right)c_j^\mathsmaller{M}(X^\vee\cap Q)h^{m-1-j}\\
 (\ast) & = 2\sum_{j=0}^m(-1)^j\left(m-j\right)\left[\sum_{i=0}^j(-1)^{j-i}2^{j-i}c_{i}^\mathsmaller{M}(X^\vee)\cdot h^{m-i}\right]\\
 & = 2\sum_{i=0}^m(-1)^ic_i^\mathsmaller{M}(X^\vee)h^{m-i}\left[\sum_{j=i}^m(m-j)2^{j-i}\right]\\
 & = 2\sum_{i=0}^m(-1)^i(2^{m+1-i}-1)c_i^\mathsmaller{M}(X^\vee)\cdot h^{m-i}\\
 & \quad-2\sum_{i=0}^m(-1)^i(m+1-i)c_i^\mathsmaller{M}(X^\vee)\cdot h^{m-i},
\end{split}
\end{align}
\endgroup

where in ($\ast$) we used (\ref{eq:chernxq}) and in the last equality we applied Lemma \ref{lem:2sum}. In the last expression obtained, since $X^\vee$ is transversal to $Q$, by (\ref{eq:AluED}) the first term coincides with $2\mathrm{EDdegree}(X^\vee)$, whereas by (\ref{eq:degdual}) the second term is equal to $2\delta_0(X^\vee)$, which vanishes because $X$ is not a hypersurface. Hence the identity (\ref{eq: 2EDdegree X not hypersurface}) is satisfied.

Now assume that $X$ is a hypersurface. The expression in the first line of (\ref{eq: computation difference Theorem 2EDdegree}) is exactly the polar class $\delta_0(X^\vee\cap Q)$. The computation (\ref{eq: computation difference Theorem 2EDdegree}) shows that the same expression is equal to
\[
2\mathrm{EDdegree}(X)-2\delta_0(X^\vee)=2\sum_{j=1}^{n-2}\delta_j(X^\vee)\,,
\]
by the identity (\ref{eq: EDdegree polar classes}). If this expression vanishes we get $\delta_j(X^\vee)=0$ for $j\ge 1$,
which is equivalent to $\delta_j(X)=0$ for $j\le n-3$. Hence the defect $\mathrm{codim}(X^\vee)-1$ is $n-3$ and $X^\vee$ is a point, namely $X$ is a hyperplane. In particular, $X^\vee\cap Q$ is empty and therefore $(X^\vee\cap Q)^\vee=\PP(V)$. In conclusion, the identity (\ref{eq: 2EDdegree X hypersurface}) is satisfied. 
Otherwise if $\delta_0(X^\vee\cap Q)\neq 0$, the identity (\ref{eq: 2EDdegree X hypersurface}) is satisfied as well, since in this case $0\neq\delta_0(X^\vee)=\deg(X)$.\qedhere
\end{proof}

One may wonder if Theorem \ref{thm:2EDdegree} remains true without transversality assumptions. The case of symmetric tensors studied in Example \ref{ex: ED poly dual Veronese} answers in the negative. Already the binary cubic case gives a counterexample. Indeed, as noticed in the introduction, for symmetric tensors the degree of $X^\vee$ is greater than the ED degree of $X$, the opposite of the general case.

\begin{corollary}\label{cor:hypersurface}
Let $X^\vee$ be a positive dimensional variety which is transversal to a smooth quadric $Q$. Then $(X^\vee\cap Q)^\vee$ is a hypersurface.
\end{corollary}
\begin{proof} The computation (\ref{eq: computation difference Theorem 2EDdegree}) shows that
\[
\delta_0(X^\vee\cap Q)=2\mathrm{EDdegree}(X)-2\delta_0(X^\vee)\,.
\]
If $(X^\vee\cap Q)^\vee$ is not a hypersurface, we get $\delta_0(X^\vee\cap Q)=0$, hence $\delta_j(X^\vee)=0$ for all $j\ge 1$, namely $\delta_j(X)=0$ for all $j\le n-3$. Hence the defect $\mathrm{codim}(X^\vee)-1$ is $n-3$ and $X^\vee$ is zero dimensional.\qedhere
\end{proof}

We recall that, if $X\subset\PP(V)$ is such that $\mathrm{codim}(X)\ge 2$, then $(X^\vee\cap Q)^\vee$ is a hypersurface by Corollary \ref{cor:constterm for cones}.

\begin{theorem}\label{thm:eqdual codim X greater than 1}
Let $X\subset\PP(V)$ be a variety such that $\mathrm{codim}(X)\ge 2$. Suppose that $X$ and $X^\vee$ are transversal to $Q$. Let $g$ be the equation of $(X^\vee\cap Q)^\vee$. Then for any data point $u\in V$, we have the identity
\[
\mathrm{EDpoly}_{X,u}(0)=g(u)
\]
up to a scalar factor. Moreover $X\subset (X^\vee\cap Q)^\vee$. 
\end{theorem}

\begin{proof}
On the one hand, by Corollary \ref{cor:constterm for cones} we have that $\mathrm{EDpoly}_{X,u}(0)=g^k$ for some positive integer $k$, hence $\deg(\mathrm{EDpoly}_{X,u}(0))=k\deg(g)$. On the other hand, comparing degrees by
Theorem \ref{thm:2EDdegree} we get $k=1$. The inclusion follows again from Corollary \ref{cor:constterm for cones}.\qedhere
\end{proof}

\section{The ED polynomial of a hypersurface}

If we restrict to the case when $X\subset\PP(V)$ is a projective hypersurface, the results found in Sections \ref{sec: study of the constant term} and \ref{sec: EDpoly in terms of dual varieties} allow us to describe completely the lowest coefficient of $\mathrm{EDpoly}_{X,u}$ for any $u\in V$.

\begin{theorem}\label{thm:eqdual}
Let $X\subset\PP(V)$ be an irreducible hypersurface, suppose that $X$ and $X^\vee$ are transversal to $Q$. Let $f$ be the equation of $X$ and $g$ be the equation of $(X^\vee\cap Q)^\vee$ when it is a hypersurface, otherwise define $g\coloneqq 1$. Then for any data point $u\in V$, we have the identity
\[
\mathrm{EDpoly}_{X,u}(0)=f^2(u)g(u)
\]
up to a scalar factor. 
\end{theorem}

\begin{proof}
If $X$ is a hyperplane, the statement follows by Corollary \ref{cor:linear subspace}. Otherwise $\deg(X)\ge 2$, $X^\vee$ is positive dimensional and $(X^\vee\cap Q)^\vee$ is a hypersurface by Corollary \ref{cor:hypersurface}.

On the one hand, by Corollary \ref{cor:constterm for cones} we have that $\mathrm{EDpoly}_{X,u}(0)=f^hg^k$ for some positive integers $h$ and $k$, hence $\deg(\mathrm{EDpoly}_{X,u}(0))=h\deg(f)+k\deg(g)$. On the other hand, by Corollary \ref{cor: degree constant term} and Theorem \ref{thm:2EDdegree} we have that $h\geq 2$ and $\deg(\mathrm{EDpoly}_{X,u}(0))=2\deg(f)+\deg(g)$.\qedhere
\end{proof}

The hypotheses of Theorems \ref{thm:eqdual codim X greater than 1} and \ref{thm:eqdual} are reasonable and agree with the principal results in the ED degree-philosophy. Anyway, in many important examples related to varieties of tensors with the Frobenius quadratic form, these hypotheses are not satisfied. A positive result is that we can relax the assumptions of transversality at least for computing the exact multiplicity of the equation of $X$ in $\mathrm{EDpoly}_{X,u}(0)$, when $X$ is a hypersurface.

\begin{proposition}\label{pro: multiplicity 2 in general}
Let $X\subset\PP(V)$ be an irreducible projective hypersurface. Then the equation of $X$ appears with multiplicity two in $\mathrm{EDpoly}_{X,u}(0)$.
\end{proposition}

\begin{proof}
Quadric hypersurfaces of $\PP(V)$ that are transversal to $X$ and $X^\vee$ form a dense open subset $U\subset\PP(S^2V)$. In particular, $Q$ is the limit of a sequence $\{Q_j\}\subset U$. Let $\mathrm{EDpoly}^{(j)}_{X,u}(t^2)$ be the ED polynomial of $X$ at $u\in V$ with respect to te quadric $Q_j$, for all $j$. By Theorem \ref{thm:eqdual}, for all $j$ we have
\[
\mathrm{EDpoly}^{(j)}_{Y,y}(0)=f^2\cdot g_j\,,
\]
where $g_j$ is the equation of $(X^\vee\cap Q_j)^\vee$. Moreover, by Corollary \ref{cor:constterm for cones} we know that $\mathrm{EDpoly}_{X,u}(0)=f^\alpha\cdot g^\beta$ for some nonnegative integers $\alpha$ and $\beta$, where $g$ is the equation of $(X^\vee\cap Q)^\vee$. In particular,
\[
f^\alpha\cdot g^\beta\cdot h=\mathrm{EDpoly}_{X,u}(0)\cdot h=\lim_{j\to\infty}\mathrm{EDpoly}^{(j)}_{X,u}(0)=\lim_{j\to\infty}f^2\cdot g_j=f^2\cdot\lim_{j\to\infty}g_j\,,
\]
for some homogeneous polynomial $h$, possibly a scalar. In particular, $\alpha\ge 2$.

We show that actually $\alpha=2$. If $\alpha\ge 3$, then $f$ divides $\lim_{j\to\infty}g_j$, that is, $f$ divides $g$ or $f$ divides $h$. It remains to show that $f$ cannot divide $g$. In particular, our claim is that $\mathrm{codim}_\R[(X^\vee\cap Q)^\vee]\ge 2$.

Consider a smooth point $z\in X^\vee\cap Q$ and the corresponding normal space $S_z\coloneqq N_z(X^\vee\cap Q)$. Assume that $l_1,\ldots,l_r$ are the linear polynomials defining $S_z$. We denote by $\overline{S}_z$ the variety defined by $\bar{l}_1,\ldots,\bar{l}_r$, where the bar means complex conjugation. If $z\in\overline{S}_z$, then $q(z-\bar{z},y)=0$ for all $y\in T_z(X^\vee\cap Q)$. In particular, $q(\bar{z},z)=q(\bar{z},z)-q(z,z)=q(\bar{z}-z,z)=0$, contradiction. This implies that $S_z\neq\overline{S}_z$ and, in turn, that $\mathrm{codim}_\R(S_z)\ge 2$. The claim follows by Definition \ref{def: dual variety}.
\end{proof}

The simplest case is when $V$ is $2$-dimensional. Let $C\colon f(x,y)=0$ be an affine plane curve which is transversal to the isotropic quadric at infinity. In this case, $Q_{\infty}=\{I,J\}$, where $I=[1,\sqrt{-1},0], J=[1,-\sqrt{-1},0]$.

Looking at Proposition \ref{pro: constterm in general}, for any $p\in C_{\mathrm{sm}}$, we have $J_{C,p}\neq\emptyset$ if and only if $T_pC=p+\langle v\rangle$ with $v\in\{(1,\sqrt{-1}),(1,-\sqrt{-1})\}$. In other words, the lowest term of $\mathrm{EDpoly}_{C,u}$ is the product of $f$ times the linear factors coming from tangent lines to $C$ meeting $I$ or $J$ at infinity.

Now we assume that the projective closure $\overline{C}$ is transversal to the isotropic quadric $\overline{Q}$ and we consider the ED polynomial of $\overline{C}$. We have already mentioned that $\mathrm{EDpoly}_{\overline{C},u}$ is in general not related with $\mathrm{EDpoly}_{C,u}$. This fact is even more clear when focusing on the lowest term of $\mathrm{EDpoly}_{\overline{C},u}$. In this case, according to Corollary \ref{cor:constterm for cones}, $\mathrm{EDpoly}_{\overline{C},u}(0)$ is up to a scalar factor the product of the homogenization of $f$ times the linear factors coming from tangent lines to $\overline{Q}$ at the points of $\overline{C}^\vee\cap\overline{Q}$.

\begin{example}
For example, consider the hyperbola $X\colon 4 x_1^{2}-9 x_2^{2}-1=0$ in Figure \ref{fig: ex hyperbola}. Given a data point $u=(u_1,u_2)\in\mathbb{C}^2$, one can verify that
\[
\mathrm{EDpoly}_{X,u}(0)=({4 {u}_{1}^{2}-9 {u}_{2}^{2}-1})^{2} (1296 {u}_{1}^{4}+2592 {u}_{1}^{2} {u}_{2}^{2}+1296 {u}_{2}^{4}-936 {u}_{1}^{2}+936 {u}_{2}^{2}+169)\,.
\]
As explained before, the second factor of the above polynomial is the product of the four pairwise conjugate lines tangent to $X$ and meeting $I$ or $J$ at infinity. On the other hand, we consider the projective closure $\overline{X}\colon 4 x_1^{2}-9 x_2^{2}-x_0^2=0$ of $X$ and we compute its ED polynomial with respect to the point $\overline{u}=[1,u_1,u_2]$. Now we obtain that
\[
\mathrm{EDpoly}_{\overline{X},\overline{u}}(0)=({4 x^{2}-9 y^{2}-1})^{2} (1024 x^{4}+2880 x^{2} y^{2}+2025 y^{4}-832 x^{2}+1170 y^{2}+169)\,.
\]
Note that the second factor of $\mathrm{EDpoly}_{\overline{X},\overline{u}}(0)$ corresponds to the dual variety of $\overline{X}^\vee\cap\overline{Q}$ which is the union of four pairwise conjugate lines different from its corresponding ones in $\mathrm{EDpoly}_{X,u}(0)$.

\begin{figure}
\centering
\begin{minipage}[c]{.40\textwidth}
\begin{tikzpicture}[line cap=round,line join=round,>=triangle 45,x=1.0cm,y=1.0cm, scale=2.5]
\draw[->, color=black] (-1.3,0.) -- (1.3,0.);
\foreach \x in {-1.2,-0.9,-0.6,-0.3,0.3,0.6,0.9,1.2}
\draw[shift={(\x,0)},color=black] (0pt,0.5pt) -- (0pt,-0.5pt) node[below] {\tiny $\x$};
\draw[->, color=black] (0.,-1.) -- (0.,1.);
\foreach \y in {-0.9,-0.6,-0.3,0.3,0.6,0.9}
\draw[shift={(0,\y)},color=black] (0.5pt,0pt) -- (-0.5pt,0pt) node[left] {\tiny $\y$};
\draw[color=black] (0pt,-2.6pt) node[right] {\tiny $0$};
\clip(-1.2,-1.) rectangle (1.2,1.);
\draw [samples=50,domain=-0.99:0.99,rotate around={0.:(0.,0.)},xshift=0.cm,yshift=0.cm,line width=0.5pt] plot ({0.5*(1+(\x)^2)/(1-(\x)^2)},{0.33333333333333337*2*(\x)/(1-(\x)^2)});
\draw [samples=50,domain=-0.99:0.99,rotate around={0.:(0.,0.)},xshift=0.cm,yshift=0.cm,line width=0.5pt] plot ({0.5*(-1-(\x)^2)/(1-(\x)^2)},{0.33333333333333337*(-2)*(\x)/(1-(\x)^2)});
\begin{scriptsize}
\draw[color=black] (1,0.4) node {$X$};
\end{scriptsize}
\end{tikzpicture}
\end{minipage}
\hspace{2cm}
\begin{minipage}[c]{.40\textwidth}
\begin{tikzpicture}[line cap=round,line join=round,>=triangle 45,x=1.0cm,y=1.0cm, scale=3]
\draw[->, color=black] (-1.,0.) -- (1.,0.);
\foreach \x in {-0.9,-0.6,-0.3,0.3,0.6,0.9}
\draw[shift={(\x,0)},color=black] (0pt,0.5pt) -- (0pt,-0.5pt) node[below] {\tiny $\x$};
\draw[->, color=black] (0.,-1.) -- (0.,1.);
\foreach \y in {-0.9,-0.6,-0.3,0.3,0.6,0.9}
\draw[shift={(0,\y)},color=black] (0.5pt,0pt) -- (-0.5pt,0pt) node[left] {\tiny $\y$};
\draw[color=black] (0pt,-2.2pt) node[right] {\tiny $0$};
\clip(-1.,-1.) rectangle (1.,1.);
\draw [color=red, line width=0.2pt,domain=-1.:1.] plot(\x,{(--0.6851599765428179--0.09989359996496527*\x)/1.1748611340918536});
\draw [color=red, line width=0.2pt,domain=-1.:1.] plot(\x,{(--0.6851601874515453-1.1748612995422902*\x)/-0.09989343451452859});
\draw [color=red, line width=0.2pt,domain=-1.:1.] plot(\x,{(--0.6851603426156989-0.09989357885759098*\x)/-1.1748614438853526});
\draw [color=red, line width=0.2pt,domain=-1.:1.] plot(\x,{(--0.6851601317069237--1.174861278434916*\x)/0.09989374430802767});
\draw [line width=0.2pt,domain=-1.:1.] plot(\x,{(--0.72222203829297-0.*\x)/1.201850272116265});
\draw [line width=0.2pt] (0.6009252125905417,-1.) -- (0.6009252125905417,1.);
\draw [line width=0.2pt,domain=-1.:1.] plot(\x,{(--0.7222224062150511-0.*\x)/-1.2018505782459408});
\draw [line width=0.2pt] (-0.6009252125905417,-1.) -- (-0.6009252125905417,1.);
\draw [rotate around={45.:(0.,0.)},line width=0.2pt] (0.,0.) ellipse (0.7071067811865478cm and 0.47140452079103157cm);
\begin{scriptsize}
\draw [fill=black] (0.6009252891229704,-0.6009252891229704) circle (0.3pt);
\draw[color=black] (0.65,-0.65) node {$D$};
\draw [fill=black] (-0.6009251360581325,0.6009251360581325) circle (0.3pt);
\draw[color=black] (-0.65,0.65) node {$A$};
\draw [fill=black] (-0.6009252891229704,-0.6009252891229704) circle (0.3pt);
\draw[color=black] (-0.55,-0.55) node {$C$};
\draw [fill=black] (0.6009251360581325,0.6009251360581325) circle (0.3pt);
\draw[color=black] (0.55,0.55) node {$B$};
\draw [fill=black] (0.5374839325138808,-0.5374839325138808) circle (0.3pt);
\draw[color=black] (0.49,-0.49) node {$G$};
\draw [fill=black] (-0.5374837670634441,0.5374837670634441) circle (0.3pt);
\draw[color=black] (-0.49,0.49) node {$E$};
\draw [fill=black] (-0.6373775113714718,-0.6373775113714718) circle (0.3pt);
\draw[color=black] (-0.69,-0.69) node {$H$};
\draw [fill=black] (0.6373773670284094,0.6373773670284094) circle (0.3pt);
\draw[color=black] (0.69,0.69) node {$F$};
\draw[color=black] (-0.28,0.28) node {$X$};
\end{scriptsize}
\end{tikzpicture}
\end{minipage}
\caption{The example of the hyperbola $X$.}\label{fig: ex hyperbola}
\end{figure}
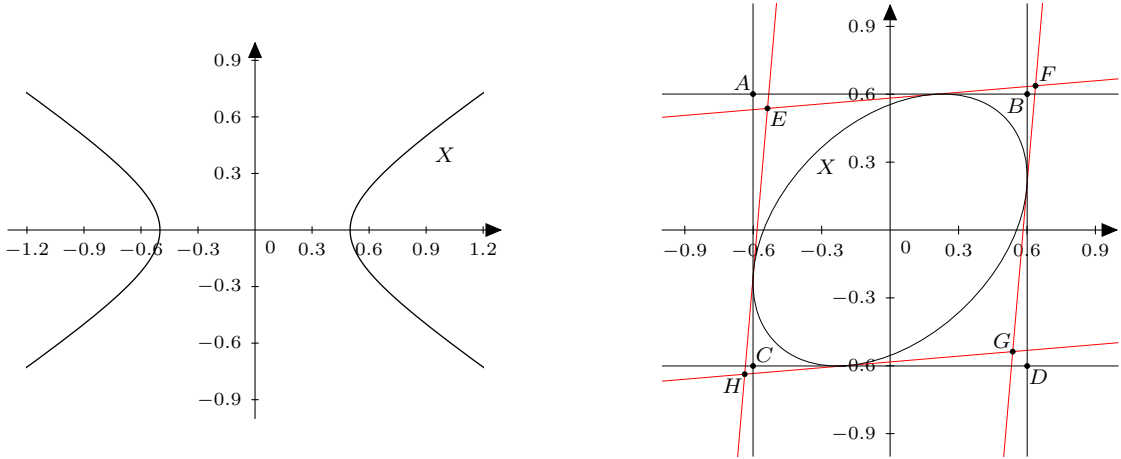

In order to display all these lines, we consider the change of coordinates of equations $z_1=-\sqrt{-1}\hspace{0.5mm}x_1, z_2=x_2+\sqrt{-1}\hspace{0.5mm}x_3, z_3=x_2-\sqrt{-1}\hspace{0.5mm}x_3.$ Then the image of $X$ is the ellipse of equation $13x^2-10xy+13y^2-4=0$. On the one hand, the points $A,B,C,D$ in Figure \ref{fig: ex hyperbola} generate the four lines corresponding to the second factor of $\mathrm{EDpoly}_{X,u}(0)$. Note that in the new coordinates the lines meeting $I$ and $J$ at infinity are the horizontal and vertical lines respectively. On the other hand, the union of the four lines generated by $E,F,G,H$ correspond to the second factor of $\mathrm{EDpoly}_{\overline{X},\overline{u}}(0)$.
\end{example}

For all $n\geq 1$ we define the integer
\begin{equation}\label{eq: def N}
N\coloneqq
\begin{cases}
2(n-1) & \mbox{if } d=2\\
d\frac{(d-1)^{n-1}-1}{d-2} & \mbox{if } d\geq 3\,.
\end{cases}
\end{equation}
If the hypersurface $X\subset\PP(V)$ is general, then by \cite[Corollary 2.10]{DHOST} $\mathrm{EDdegree}(X)=N$. For example, a general plane curve $C\subset\PP^2$ has $\mathrm{EDdegree}(C)=d^2$ and a general surface $S\subset\PP^3$ has $\mathrm{EDdegree}(S)=d(d^2-d+1)$. 

Things get more difficult if we allow $X$ to have isolated singularities. Recalling from \cite[Section 1.2.3]{D} that $\mu(Y,y)$ is the {\em Milnor number} of an isolated singularity $y$ of a complete intersection subvariety $Y\subset\PP(V)$, we have the following result (see \cite{Pie15}).
\begin{theorem}
Let $X\subset\PP(V)\cong\PP^{n-1}$ be a hypersurface of degree $d$. Suppose $X$ has only isolated singularities. For any point $x\in X_{\mathrm{sing}}$, let
\[
e(X,x)\coloneqq\mu(X,x)+\mu(H\cap X,x)\,,
\]
where $H$ is a general hyperplane section of $X$ containing $x$. Then
\begin{equation}\label{eq: EDdegree hypersurface with isolated singularities}
\mathrm{EDdegree}(X)=N-\sum_{x\in X_{\mathrm{sing}}}e(X,x)\,,
\end{equation}
where the integer $N$ has been defined in (\ref{eq: def N}).
\end{theorem}
It is known (see \cite[Example 1.2.3]{D}) that if $x\in X$ is a singular point of type $A_k$, then
\[
\mu(X,x)=k,\quad \mu(H\cap X,x)=1\,.
\]
This gives the formula of the ED degree of a hypersurface with $s$ singularities of type $A_{k_1},\ldots,A_{k_s}$
\[
\mathrm{EDdegree}(X)=N-(k_1+1)-\cdots-(k_s+1)\,.
\]
In particular, if we consider a plane curve $C$ of degree $d$ with with $\delta$ ordinary nodes and $\kappa$ ordinary cusps, the ED degree of $C$ is
\begin{equation}\label{eq: ED degree curve nodes cusps}
\mathrm{EDdegree}(C)=d^2-2\delta-3\kappa\,.
\end{equation}

\begin{example}\label{ex: cardioid}
Let $C\subset\R^2$ be the real affine cardioid of equation $(x^2+y^2-2x)^2-4(x^2+y^2)=0$. It has a cusp at the origin and at the isotropic points at $H_{\infty}$. Hence, according to the formula (\ref{eq: ED degree curve nodes cusps}), $\mathrm{EDdegree}(\overline{C})=16-3\times 3=7$. On the other hand, $\mathrm{EDdegree}(C)=3$: the drop is caused essentially by the non-transversality with $Q_{\infty}$.
Computing the ED polynomial of $C$, one may observe that its leading coefficient is $(x-1)^2+y^2$, namely one branch of the distance function diverges when the chosen data point is $u=(1,0)$. It is interesting to note that the projective embedding of $(1,0)$ is the meeting point of the three tangent cones at the three cusps of $\overline{C}$. Moreover, the cardioid $C$ is the trace left by a point, initially at the origin, on the perimeter of a circle of radius $1$ that is rolling around the circle centered in $(1,0)$ of the same radius. In Figure \ref{fig:cardioid} we see that the ED discriminants $\Sigma_C$ and  $\Sigma_{\overline{C}}$ are dramatically different. While the real part of $\Sigma_C$ is again a cardioid, the real part of $\Sigma_{\overline{C}}$ divides the plane in four connected components, one of them is shown in the detail on the right of Figure  \ref{fig:cardioid}.

\begin{figure}[htbp]
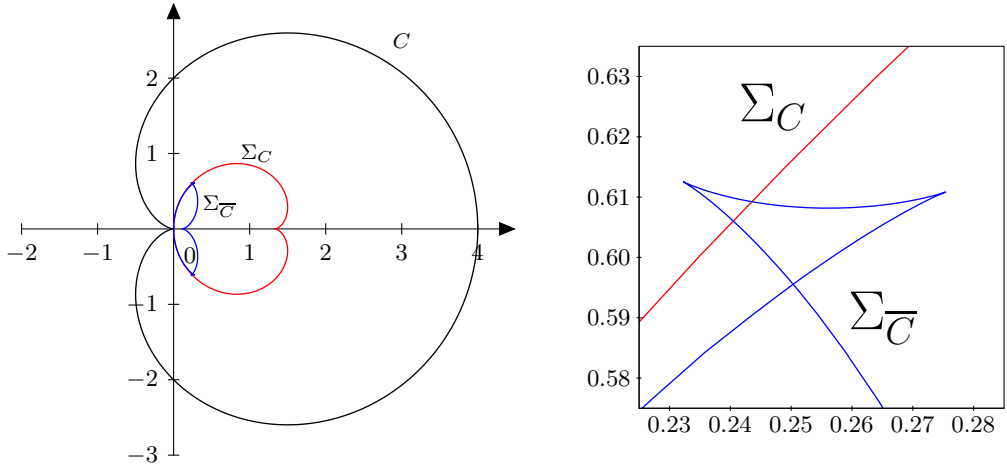

\centering
\begin{minipage}[c]{.40\textwidth}

\end{minipage}
\caption{The cardioid $C$ with its ED dscriminant $\Sigma_C$ and the restriction to the affine plane of the ED discriminant $\Sigma_{\overline{C}}$ of its projectivization $\overline{C}$. A detail of $\Sigma_{\overline{C}}$ on the right.}\label{fig:cardioid}
\end{figure}
\end{example}

A remarkable example when formula (\ref{eq: EDdegree polar classes}) cannot be used is the case of symmetric tensors of fixed degree.

\begin{example}[The ED polynomial of the dual of the Veronese variety]\label{ex: ED poly dual Veronese}
First we clarify our notation. Let $W_\R$ be a real $n$-dimensional Euclidean space endowed with a positive definite quadratic form $\tilde{q}$ and let $W\coloneqq W_\R\otimes\C$ be its associated complex vector space. Define $V_\R\coloneqq\mathrm{Sym}^dW_\R$. We set $x_1,\ldots,x_n$ and $a_{i_1\cdots i_{n}}$ with $i_1+\cdots+i_n=d$ as coordinates of $W_\R$ and $V_\R$, respectively. Each element $f\in V=V_\R\otimes\C$ can be interpreted, up to scalars, as a homogeneous polynomial $f\in\C[x_1,\ldots,x_n]_d$, or equivalently as a degree $d$ symmetric tensor on $W$. The image of the map
\[
v_{n,d}\colon W_\R\to V_\R,\quad v_{n,d}(x_1,\ldots,x_n)\coloneqq(x_1^{i_1}\cdots x_n^{i_n})_{i_1+\cdots+i_n=d}
\]
is the cone over the {\em $d$-th Veronese variety} of $\PP(W_\R)$. We will denote this cone by $X_\R\subset V_\R$ and its complexification by $X\subset V$. The variety $X$ consists of $d$-th powers of linear polynomials or, in other words, rank one symmetric tensors on $W$. Its dual $X^\vee$ is the well-known {\em discriminant hypersurface}, namely the variety defined by the vanishing of the classical multivariate discriminant $\Delta_d$ of a degree $d$ homogeneous polynomial on $W$ (see \cite[\rom{13}]{GKZ}).

If we make $V_\R$ an Euclidean space with a positive definite quadratic form $q$ that creates a transverse intersection between $X$ and the isotropic quadric $Q$, we have the identity (see \cite[Proposition 7.10]{DHOST})
\begin{equation}\label{eq: ED degree Veronese variety}
\mathrm{EDdegree}(X)=\frac{(2d-1)^n-(d-1)^n}{d}\,.
\end{equation}
Nevertheless, we want the Euclidean distance on $V_\R$ to be compatible with the action of the group $SO(W)$. To this end, the good choice for $Q\subset\PP(V)$ is the following:
\[
Q\colon\sum_{i_1+\cdots+i_n=d}\binom{d}{i_1,\ldots,i_n}a_{i_1\cdots i_n}^2=0\,.
\]
With this choice, one verifies immediately that $q(v_{n,d}(x))=\tilde{q}(x)^d$ for all $x\in W$. However, this choice causes a drop in the ED degree of $X$ (see \cite[Corollary 8.7]{DHOST}):
\begin{equation}\label{eq: ED degree Veronese variety, good quadratic form}
\mathrm{EDdegree}(X)=\frac{(d-1)^n-1}{d-2}\,.
\end{equation}
Indeed, in this case $X\cap Q$ is a non-reduced variety of multiplicity $d$. This non-transversality is confirmed by the fact that the identity of Theorem \ref{thm:2EDdegree} cannot hold, since the integer $c_{n,d}=\mathrm{deg}(X^\vee)-\mathrm{EDdegree}(X^\vee)$ vanishes for $d=2$ and is positive when $d>2$ for all $n\geq 1$, as pointed out by the second author in \cite{Sod}.

The ED polynomial of the hypersurface $X^\vee$ is studied in detail in \cite{Sod}. In this setting, given $f\in V_\R$, any critical point of the Euclidean distance function $d_f$ on $X_\R$ is called a \emph{critical rank one symmetric tensor} for $f$. The case $d=2$ deals with real symmetric matrices. For any symmetric matrix $U\in V$, we have the identity
\begin{equation}\label{eq: EDpoly d=2}
\mathrm{EDpoly}_{X^\vee,U}(t^2)=\psi_U(t)\psi_U(-t)=\det(U-tI_n)\det(U+tI_n)\,,
\end{equation}
where $\psi_U(t)=\det(U-tI_n)$ is the characteristic polynomial of $U$. In particular, the lowest term of $\mathrm{EDpoly}_{X^\vee,U}$ is the square of $\det(U)$. The classic notions of eigenvalue and eigenvector of a symmetric matrix have been extended by Lek-Heng Lim and Liqun Qi in \cite{L,Q} to any value of $d$ (and, more in general, for any $n$-dimensional tensor of order $d$ non necessarily symmetric). In particular, they introduced the {\em E-eigenvectors} and {\em E-eigenvalues} of a symmetric tensor. We follow the notation used in \cite[Definition 1.1]{Sod}.

A striking interpretation of the critical rank one symmetric tensors for $f\in V_\R$ is contained in the following result.

\begin{theorem}[\cite{L}, variational principle]\label{thm: Lim}
Given $f\in V_\R$, the critical rank one symmetric tensors for $f$ are exactly of the form $x^d$, where $x=(x_1,\ldots,x_n)$ is an eigenvector of $f$.
\end{theorem}

A consequence of this result is that every symmetric tensor $f\in V$ has at most $\mathrm{EDdegree}(X)$ (see formula (\ref{eq: ED degree Veronese variety, good quadratic form})) distinct E-eigenvalues when $d$ is even, and at most $\mathrm{EDdegree}(X)$ pairs $(\lambda,-\lambda)$ of distinct E-eigenvalues when $d$ is odd. This bound is attained for general symmetric tensors (see \cite{CS}).

The characterization of critical rank one symmetric tensors given in Theorem \ref{thm: Lim} is used by Draisma, Ottaviani and Tocino in \cite{DOT}, where they deal more in general with the \emph{best rank $k$ approximation problem} for tensors. In particular, it is shown in \cite{DOT} that the critical rank one tensors (more generally, the critical rank $k$ tensors) are contained in a subspace. This fact is false for general varieties, because looking at Example \ref{ex: ED polynomial affine conic}, one can verify immediately that the four critical points of an ellipse are not aligned.

As the eigenvalues of a symmetric matrix are the roots of its characteristic polynomial, the E-eigenvalues of a symmetric tensor $f\in V$ can be computed via its {\em E-characteristic polynomial} $\psi_f$ (see \cite[Definition 1.2]{Sod}). When $d$ is even, the E-characteristic polynomial of $f=f(x)\in V_\R$ is, up to a scalar,
\begin{equation}\label{eq: characteristic polynomial discr even case}
\psi_f(t)\coloneqq\Delta_d\left(f(x)-t\tilde{q}(x)^{d/2}\right)\,,
\end{equation}
where we recall that $\tilde{q}$ is the quadratic form of $W_\R$. On the other hand, a relation equivalent to (\ref{eq: characteristic polynomial discr even case}) is no longer possible for $d$ odd.

It is shown in \cite{Sod} that, for any value of $d$, the lowest term of $\mathrm{EDpoly}_{X^\vee,f}(t^2)$ is the square of $\Delta_d(f)$, namely the polynomial defining the hypersurface $X^\vee$. Therefore the integer $c_{n,d}$ introduced previously is exactly the degree of the leading coefficient of $\mathrm{EDpoly}_{X^\vee,f}(t^2)$. Hence the distance function $d_f$ on $X^\vee$ is not integral for all $n\geq 1$ and $d>2$. In particular, the hypersurface cut out by the leading coefficient of $\mathrm{EDpoly}_{X^\vee,f}(t^2)$ is the dual of the reduced variety associated to $X\cap Q$, that is the dual of the embedding in $V_\R$ of the isotropic quadric of $W_\R$. These facts combined together lead to a closed formula (see \cite[Main Theorem]{Sod}) for the product of the E-eigenvalues of a symmetric tensor, which generalizes to the class of symmetric tensors the known fact that the determinant of a symmetric matrix is the product of its eigenvalues.

In particular, the ED polynomial of $X^\vee$ has a nice factorization which generalizes the formula (\ref{eq: EDpoly d=2}) to an arbitrary even degree:
\[
\mathrm{EDpoly}_{X^\vee,f}(t^2)=\psi_f(t)\psi_f(-t)=\Delta_d\left(f(x)-t\tilde{q}(x)^{d/2}\right)\Delta_d\left(f(x)+t\tilde{q}(x)^{d/2}\right)\,.
\]
Properties like the above-mentioned one are studied in a paper in preparation by the second author in the more general setting of non-symmetric tensors and in particular in the case of boundary format tensors.
\end{example}

\section{The case of matrices}
Consider the vector space $M_{m,n}$ of $m\times n$ matrices with $m\leq n$, with natural bilinear form defined by $q(U,T)\coloneqq\mathrm{tr}(UT^\mathsmaller{T})$ for all $U,T\in M_{m,n}$. This form induces the Frobenius norm $q(U)^2=\sum_{i,j}u_{ij}^2$. For any data point $U\in M_{m,n}$, we can consider its {\em singular value decomposition}
\begin{equation}\label{eq: SVD of U}
U=T_1\cdot\Sigma\cdot T_2\,,
\end{equation}
where $\Sigma=\mathrm{diag}(\sigma_1,\ldots,\sigma_m)$ and $\sigma_1>\ldots>\sigma_m$ are the singular values of $U$, while $T_1, T_2$ are orthogonal matrices of format $m\times m$ and $n\times n$ respectively. For any $r\in\{1,\ldots,m\}$, let $X_r\subset M_{m,n}$ be the variety of $m\times n$ matrices of rank $\leq r$. The Eckart-Joung Theorem states that the critical points of the distance function $d_U$ from $X_r$ are of the form
\[
U=T_1\cdot(\Sigma_{i_1}+\cdots+\Sigma_{i_r})\cdot T_2\,,
\]
where $\Sigma_j=\mathrm{diag}(0,\ldots,0,\sigma_j,0,\ldots,0)$ for all $j\in\{1,\ldots,m\}$, and $I=\{i_1<\cdots<i_r\}$ runs over all $r$-element subsets of $\{1,\ldots,m\}$. In particular we have that $\mathrm{EDdegree}(X_r)=\binom{m}{r}$. With a bit of computation, applying Proposition \ref{pro:roots charpoly} and using equation (\ref{eq: SVD of U}), the ED polynomial of $X_r$ can be written (up to a constant factor) as (see also \cite[Example 2.9]{HW})
\begin{align*}
\mathrm{EDpoly}_{X_r,U}(t^2) & = \prod_{1\leq i_1<\cdots<i_r\leq m}\left(t^2-\mathrm{tr}\{[\Sigma-(\Sigma_{i_1}+\cdots+\Sigma_{i_r})]^2\}\right)\\
& = \prod_{1\leq i_1<\cdots<i_r\leq m}\left[t^2-(\sigma_1^2+\cdots+\widehat{\sigma_{i_1}^2}+\cdots+\widehat{\sigma_{i_r}^2}+\cdots+\sigma_m^2)\right].
\end{align*}
In particular, for $r=m-1$ the last formula simplifies as in the following result. We recall that the dual variety of $X_r$ is $(X_r)^\vee=X_{m-r}$ for all $r\in\{1,\ldots,m-1\}$.
\begin{theorem}\label{thm: EDpoly matrices}
Let $X=X_1\subset M_{m,n}$ be the cone of rank one matrices. Then
\begin{align*}
\mathrm{EDpoly}_{X^\vee,U}(t^2) & =\det\left(UU^\mathsmaller{T}-t^2 I_m\right)\\
\mathrm{EDpoly}_{X,U}(t^2) & =(-1)^m\det\left[t^2I_m+UU^\mathsmaller{T}-\mathrm{tr}(UU^\mathsmaller{T})I_m\right].
\end{align*}
\end{theorem}

\begin{example}\label{exa:ortoinvariant} Let ${\mathcal M}\subset{\mathbb R}^{m\times n}$ be an orthogonally invariant matrix variety as in \cite{DLOT} and let $S=\{x\mid\mathrm{Diag}(x)\in {\mathcal M}\}$ its diagonal restriction.
Without loss of generality, suppose $m\leq n$.
Let $\sigma(U)=\left(\sigma_1(U),\ldots \sigma_m(U)\right)$ be the set of singular values of a matrix $U\in{\mathbb R}^{m\times n}$.

Then the Theorem 4.11 of \cite{DLOT} may be reformulated as follows.
$$\mathrm{EDpoly}_{{\mathcal M},U}(t^2)=\mathrm{EDpoly}_{S,\sigma(U)}(t^2).$$
Note that this correspondence preserves duality, that is
$$\mathrm{EDpoly}_{{\mathcal M}^\vee,U}(t^2)=\mathrm{EDpoly}_{S^\vee,\sigma(U)}(t^2).$$

For example, the ED polynomial of the essential variety ${\mathcal E}\subset{\mathbb C}^{3\times 3}$ in computer vision (see
\cite[Example 2.7]{DLOT}) is
\[
\mathrm{EDpoly}_{{\mathcal E},U}(t^2)=\prod_{(ijk)\textrm{\ cyclic perm. of\ }\{1,2,3\}}\left[t^2-(\sigma_i-\sigma_j)^2/2-\sigma_k^2\right]\left[t^2-(\sigma_i+\sigma_j)^2/2-\sigma_k^2\right]\,,
\]
where $\sigma_1$, $\sigma_2$, $\sigma_3$ are the singular values of the $3\times 3$ matrix $U$. Note that neither ${\mathcal E}$ nor ${\mathcal E}^\vee$ is a hypersurface, while the varieties $\left({\mathcal E}^\vee\cap Q\right)^\vee$ and $\left({\mathcal E}\cap Q\right)^\vee$ are both reduced. Therefore $\left({\mathcal E}^\vee\cap Q\right)^\vee$ and $\left({\mathcal E}\cap Q\right)^\vee$ are both hypersurfaces by Corollary \ref{cor:constterm for cones}.

The lowest term $\mathrm{EDpoly}_{{\mathcal E},U}(0)$,
which is the equation of $\left({\mathcal E}^\vee\cap Q\right)^\vee$, is equal to
\[
4a_1^6-12a_1^4a_2-15a_1^2a_2^2+144a_1^3a_3-4a_2^3-90a_1a_2a_3-27a_3^2\,,
\]
where $a_1, a_2, a_3$ are defined by the equation
$\det(UU^\mathsmaller{T}+\lambda I)=\lambda^3+\lambda^2a_1+\lambda a_2+a_3$.
The lowest term of the ED polynomial for the dual variety, namely $\mathrm{EDpoly}_{{\mathcal E}^\vee,U}(0)$, is equal to the discriminant
\[
a_1^2a_2^2-4a_1^3a_3-4a_2^3+18a_1a_2a_3-27a_3^2\,,
\]
which is the equation of $\left({\mathcal E}\cap Q\right)^\vee$.
\end{example}

\section*{Acknowledgements}
We thank Paolo Aluffi who showed us how Theorem \ref{thm:2EDdegree}, that we proved originally in the smooth case, can be extended to any variety using Chern-Mather classes.
The first author is indebted to J.~Draisma, E.~Horobe\c{t}, B.~Sturmfels, R.~Thomas, his coauthors of \cite{DHOST}, where the EDphilosophy was introduced first.
Both authors are members of INDAM-GNSAGA.

\end{document}